\documentclass[11pt]{amsart}
\usepackage{amscd,amsmath,amssymb,amsfonts,verbatim}
\usepackage[cmtip, all]{xy}
\usepackage{MnSymbol}
\usepackage{xcolor}
\usepackage{tikz}
\usepackage{tikz-cd}
\usepackage{comment}

\usepackage[colorlinks]{hyperref}
\hypersetup{
  colorlinks=true,
  citecolor=red,
  linkcolor=blue,
  urlcolor=cyan}

\newtheorem{thm}{Theorem}[section]
\newtheorem{prop}[thm]{Proposition}
\newtheorem{lem}[thm]{Lemma}
\newtheorem{cor}[thm]{Corollary}
\newtheorem{conj}[thm]{Conjecture}
\theoremstyle{definition}

\newtheorem{defn}[thm]{Definition}
\theoremstyle{remark}
\newtheorem{remk}[thm]{Remark}
\newtheorem{remks}[thm]{Remarks}

\newtheorem{exm}[thm]{Example}
\newtheorem{exms}[thm]{Examples}
\newtheorem{notat}[thm]{Notation}
\numberwithin{equation}{section}

\newenvironment{Def}{\begin{defn}}%
{\hfill$\square$\end{defn}}
\newenvironment{rem}{\begin{remk}}%
{\hfill$\square$\end{remk}}
{\hfill$\square$\end{remks}}
{\hfill$\square$\end{exm}}
{\hfill$\square$\end{exms}}
{\hfill$\square$\end{notat}}

\newcommand{\thmref}{Theorem~\ref}
\newcommand{\propref}{Proposition~\ref}

\newcommand{\lemref}{Lemma~\ref}

\newcommand{\sX}{{\mathcal X}}

\newcommand{\sZ}{{\mathcal Z}}

\newcommand{\C}{{\mathbb C}}

\newcommand{\F}{{\mathbb F}}

\newcommand{\M}{{\mathbb M}}

\renewcommand{\P}{{\mathbb P}}
\newcommand{\Q}{{\mathbb Q}}

\newcommand{\Z}{{\mathbb Z}}

\newcommand{\ff}{{\mathfrak f}}

\newcommand{\Ker}{{\rm Ker}}

\newcommand{\CH}{{\rm CH}}

\newcommand{\inj}{\hookrightarrow}
\newcommand{\red}{{\rm red}}

\newcommand{\Hom}{{\rm Hom}}

\newcommand{\Spec}{{\rm Spec \,}}
\newcommand{\sing}{{\rm sing}}

\newcommand{\Char}{{\rm char}}
\newcommand{\cdh}{{\rm cdh}}

\newcommand{\Sch}{{\operatorname{\mathbf{Sch}}}}

\newcommand{\Sm}{{\mathbf{Sm}}}

\newcommand{\dm}{{\mathbf{DM}}}

\newcommand{\et}{{\text{\'et}}}

\newcommand{\ds}{{/\kern-3pt/}}

\newcommand{\Tor}{{\operatorname{Tor}}}

\newcommand{\un}{\underline}

\renewcommand{\dim}{\text{\rm dim}}

\newcommand{\tuborg}{\left\{\begin{array}{ll}}
\newcommand{\sluttuborg}{\end{array}\right.}

\newcommand{\zar}{{\rm zar}}
\newcommand{\nis}{{\rm nis}}

\newcommand{\reg}{{\rm reg}}

\newcommand{\tor}{{\rm tor}}

\newcommand{\etl}{{\acute{e}t}}

\newcommand{\hs}{\heartsuit}

\newcounter{elno}

\newcounter{elno-abc}   

\newcounter{elno-abc-prime}

\begin{document}

\title{Motivic Cohomology and K-groups of varieties over higher local fields}
\author{Rahul Gupta, Amalendu Krishna, Jitendra Rathore}
\address{Institute of Mathematical Sciences, A CI of Homi Bhabha National Institute, 4th Cross St., CIT Campus, Tharamani, Chennai,
  600113, India.} 
\email{rahulgupta@imsc.res.in}
\address{Department of Mathematics at the University of California, Santa Barbara, California, USA.}
\email{amalenduk@ucsb.edu}
\address{Department of Mathematics at the University of California, Santa Barbara, California, USA.}
\email{jitendra@ucsb.edu}

\keywords{Local fields, 0-cycles, Algebraic $K$-Theory, Milnor $K$-theory}

\subjclass[2020]{Primary 14C25, 14F22; Secondary 14F30, 19D45}

\maketitle

\begin{quote}\emph{Abstract.}
    For quasi-projective varieties over a higher local field $k_N$, we prove that its $K$-groups, above a suitable degree, are divisible-by-finite. 
   We also prove the finiteness of the prime-to-$p$ torsion subgroup of certain higher Chow groups for smooth projective varieties over such fields, where $p$ denotes the final residue characteristic of $k_N$. As an application, we show that the kernel of the tame reciprocity map is uniquely $p'$-divisible. A key ingredient in achieving these results is the finiteness of \'etale cohomology groups over such fields.

\end{quote}
\setcounter{tocdepth}{1}
\tableofcontents

\section{Introduction}\label{sec:Intro}

\subsection{Background} The algebraic $K$-groups of an algebraic variety $X$ are important invariants attached to $X$ and they encode deep arithmetic and geometric information about $X$. One of the central themes in algebraic $K$-theory is to understand these groups for nice classes of algebraic varieties. These groups are very difficult to compute, even for fields. It was a remarkable achievement when Quillen \cite{Quillen-Annals} computed all $K$-groups of finite fields.

One of the deep conjectures in the subject involves the Bass conjecture and the Parshin conjecture. The former predicts that for a smooth projective variety $X$ over a finite field, the algebraic $K$-groups $K_{i}(X)$ are finitely generated for all $i \geq 0$, while the latter predicts that $K_{i}(X)$ are torsion for $i > 0$. Over several decades, many formulations in terms of motivic cohomology, as well as their relations to other known conjectures, have been established \cite{Geisser-Ktheory}, \cite{Geisser-EMS}, \cite{Geisser-Commentarii}, \cite{Kahn-ENS}. These conjecture are far from being resolved, except in dimensions at most one (\cite[Chapter 6, Theorem 6.1]{Weibel-K}),  but there are some partial results in higher dimension (see \cite{Soule}, \cite{Kahn-ENS}).

Passing from finite fields to local fields, in contrast, the $K$-groups are not necessarily finitely generated. In \cite{Tate-Kyoto}, Tate showed that $K_{2}(F)$ is the direct sum of a finite group and a divisible
group, where $F$ is a non-archimedean local field. This was later appropriately extended to higher algebraic $K$-groups for non-archimedean local fields (see \cite[Chapter VI, Section 7]{Weibel-K}). In the archimedian setting, the structure of $K$-groups for the archimedean local fields $\mathbb{R}$ and $\mathbb{C}$ has also been established (see \cite[Chapter VI, Theorem 1.6, Theorem 3.1]{Weibel-K}). For smooth varieties over $\C$,
Pedrini-Weibel \cite{Weibel-Pedrini-surface}, \cite{Weibel-Pedrini} proved that algebraic $K$-groups of a smooth variety $X$ over $\C$, above the dimension of $X$, are divisible-by-finite. 
In other words, they proved that $K_n(X)$ is the direct sum of a finite group and a divisible group when $n > {\rm dim}(X)$.

In this article, we study the $K$-groups of quasi-projective varieties over finite fields, non-archimedean local fields, and more generally, higher local fields (in the sense of Kato), and establish analogous divisible-by-finite results. The aforementioned results of Pedrini and Weibel are among the primary motivations for our work. We now state our main results.

\subsection{Algebraic $K$-groups for quasi-projective varieties over $N$-local fields} Recall that for $N \geq 0$,  an $N$-local field is defined inductively so that  it is
a complete discrete valued field $k_N$ such that its residue field $k_{N-1}$ is an $(N-1)$-local field, where $0$-local field is a finite field. 
In particular, $k_N$ is an $N$-local field if and only if there exists an $(N+1)$-tuple  $(k_N, k_{N-1}, \dots, k_1, k_0)$ of fields such that each $k_i$ is a complete discrete valuation field, $k_{i-1}$ is the residue field of $k_i$, and $k_0$ is a finite field.
We let $p>0$ denote the  characteristic of $k_0$. 
Our first main result of this paper is the following.

\begin{thm}\label{thm:Main:structure of algebraic-K-group over N-local}
    Let $X$ be a connected quasi-projective variety of dimension $d\geq 0$ over $k_{N}$. Then for any integer $m \geq d+N+1$, we have
    \begin{center}
        $K_{m}(X) \cong F \oplus D$,
    \end{center}
    where $F$ is a finite group and $D$ is  a uniquely $p'$-divisible group. In particular, $K_{m}(X)_{\tor}$ (except $p$-torsion) is finite. Moreover, if ${\rm char}(k_N) = {\rm char}(k_0)$, then $D$ is a uniquely divisible group and $K_{m}(X)_{\tor}$ is finite. 
\end{thm}

In the case where $X$ is a smooth projective variety over a finite field, we get following corollary, which shows the equivalence of Bass and Parshin's Conjectures in higher degrees. 
 \begin{cor}\label{cor:main:equivalence of conjectures}
      	Let $X$ be a smooth projective variety of dimension $d \geq 0$ over a finite field $k$. For any integer $m \geq  d + 1$, the following are equivalent.
\begin{enumerate}
    \item $K_{m}(X)$ is a finite group of order not divisible by $p$.
    \item  $K_{m}(X)$ is finitely generated.
    \item $K_{m}(X)$ is torsion.
\end{enumerate}
     
  Moreover, if $K_{d}(X)$ is torsion, then $K_{d}(X)$ is finite.
      \end{cor}

 For smooth projective varieties over a finite field $k_{0}$ (i.e., the case $N = 0$), the group $D$ in \thmref{thm:Main:structure of algebraic-K-group over N-local} is conjecturally zero. This stronger version of the theorem was established for $d = 0$ by Quillen \cite{Quillen-Annals} and for $d = 1$ by Harder and Quillen (see \cite[Chapter VI, Theorem 6.1]{Weibel-K}). In higher dimensions, $K$-groups are known to be torsion for certain classes of surfaces and threefolds (see \cite{Gregory-JPA}, \cite[Théorème 6]{Soule}, and \cite[Corollaire 2.1]{Kahn-ENS}). Explicit examples of such smooth projective surfaces include abelian surfaces, Kummer surfaces, and supersingular $K3$ surfaces \cite[Proposition A.5]{Gregory-JPA}. Hence by Theorem \ref{thm:Main:structure of algebraic-K-group over N-local}, one can conclude that $D$ is zero for these classes of varieties when $m$ is greater than the dimension of  variety.
 When $N = 1$, the case $d = 0$ is well-understood (see \cite{Tate-Kyoto}, \cite{Merkurjev} and \cite[Chapter VI, Theorem 7.3]{Weibel-K}). Furthermore, known results in class field theory for smooth projective varieties over local fields have partially aided the study of $K$-groups \cite{Bloch-Annals}, \cite{Saito-JNT}, \cite{Forre-Crelle} through an understanding of $SK_{1}(X)$. By contrast, very little is known when $N \geq 2$ (see \cite{Fesenko-Invitation}).

\medskip
We now state analogue results for \'etale $K$-groups of smooth projective varieties over finite and local fields, which were obtained as a consequence of proving similar results for algebraic $K$-groups. For $i \ge 0$, the \'{e}tale $K$-group  $K^{\et}_{i}(X)$ is defined as the $i^{th}$ homotopy
group of the spectrum $\widehat{\mathbf{K}}^{\et}_{X}$, which is defined in
\cite[Definition~14.3]{Dwyer-Friedlander-EtaleK} (also see \cite{Friedlander}). For $v \ge 1$ and $i \ge 0$, the $i^{th}$ \'{e}tale $K$-group with coefficients
$\mathbb{Z}/l^{v}$, denoted by $K^{\et}_{i}(X,\mathbb{Z}/l^{v})$, is defined
 as the $i^{th}$ homotopy group of the spectrum
$\widehat{\mathbf{K}}^{\et}_{X} \wedge \mathcal{M}(l^{v})$, where
$\mathcal{M}(l^{v})$ denotes the mod $l^{v}$ Moore spectrum
(see \cite[Definition~4.3]{Dwyer-Friedlander-EtaleK}), where $l$ is a prime invertible on $X$. We have the following.

\begin{thm}\label{thm:main:structure of etale-K-group over N-local for smooth}
    Let $X$ be a smooth projective  variety of dimension $d \geq 0$ over a field $k$.
\begin{enumerate}
    \item If $k$ is finite, then for any integer $m \geq 1$, we have
    \begin{center}
        $K^{\et}_{m}(X) \cong F \oplus D$,
    \end{center}
    where $F$ is finite and $D$ is a uniquely $p'$-divisible group.

\vspace{2mm}
     \item If $k$ is a non-archimedian local field and $X$ admits a good reduction, then for any integer $m \geq 2$, we have
    \begin{center}
        $K^{\et}_{m}(X) \cong F' \oplus D'$,
    \end{center}
    where $F'$ is a finite group and $D'$ is a uniquely $p'$-divisible group.
\end{enumerate}  
\end{thm}

One of the key results used to prove Theorems~\ref{thm:Main:structure of algebraic-K-group over N-local} and~\ref{thm:main:structure of etale-K-group over N-local for smooth} is to establish the finiteness results for the \'{e}tale cohomology of quasi-projective varieties over higher local fields. These results are themselves of independent importance, with a wide range of applications. We now state our results for the \'{e}tale cohomology of quasi-projective varieties.

\subsection{Finiteness of étale cohomology groups over $N$-local fields}

The finiteness of the étale cohomology groups with divisible coefficients has been a central problem (for example, see \cite[Section~2.1]{CTSS}, 
\cite{Kahn-Jnt} and \cite[Section~7]{GKR}). We generalize these results to quasi-projective varieties over higher local fields. We proved the following.

\begin{thm}\label{thm:finiteness for quasiprojective over N-local fields}
    Let $X$ be a quasi-projective variety of dimension $d \geq 0$ over an $N$-local field $k_{N}$. Then there exists an integer $M$ such that
\[
	 \lvert H_{\et}^{i} (X, \mathbb{Z}/m(n) ) \rvert \leq M
	\]   for all integers $m$ such that $(m, p) = 1$  if any of the following holds.
    \begin{enumerate}
        \item $i \nin [2n-d-N, n+d+2]$;
        \item $n \nin [0, d+N]$.
    \end{enumerate}
\end{thm}

For $N=0$, the theorem follows from \cite[Theorem~2]{Kahn-Jnt} and for $N=1$, 
the part~$(2)$ of the theorem is known in the case when $X$ is projective by \cite[Corollary~7.11]{GKR}.

  \subsection{Bloch's Higher Chow groups}

        	We now study the finiteness of torsion subgroup in higher Chow groups for smooth projective varieties over finite, local, or more generally higher local fields. Bass conjectured that algebraic $K$-groups of smooth projective varieties over finite fields are finitely generated. A more stronger conjecture predicts that
        	higher Chow groups of such varieties are also finitely generated abelian groups.
  A weaker form of the above conjecture is as follows. 
    \vspace{1mm}

        	\begin{conj}\label{higher chow are f.g.}
        			 Let $X$ be a smooth projective variety of dimension $d \geq 0$ over a finite field $k$. Then $\CH^{i}(X,j)_{\tor}$ is finite for all $i, j \geq 0$.
        	\end{conj}

            First, note that if $ i > d + j$, then $\CH^{i}(X,j) = 0$ by dimension argument. So only interesting range is $i \leq d + j$. Not much is known about this conjecture in general except two cases $j = 0$, $i= d$ and $ j = 1, i = d +1$. In first case when $j = 0$, the groups $\CH^{d}(X,0) = \CH^{d}(X) \cong \CH_{0}(X)$ are finitely generated as shown by Kato-Saito using unramified class field theory (see \cite{Kato-Saito-Unramified}). In the second case, $\CH^{d+1}(X,1)$ is shown to be  finite by Akthar (see \cite[Theorem 1.3]{Akhtar}).  Here, we prove the following. 
        	
\begin{thm}\label{finiteness for chow groups over finite fields}
        		Let $X$ be a smooth projective variety of dimension $d  \geq 0$ over a finite field $k$. Then $\CH^{i}(X,j) \lbrace p' \rbrace$ is finite for $  i \geq d+1$ and $j \geq 0$. 
        	\end{thm}
        	
 We will now discuss some results for varieties over non-archimedian local fields. We denote by $k$ a local field with the residue characteristic $p>0$. Conjecture \ref{higher chow are f.g.} is no longer true for smooth projective varieties over $k$. In \cite{AS}, Asakura and Saito construct an example of smooth projective curves $C$ over $p$-adic fields such that $\CH^{2}(C, 1) \lbrace p' \rbrace $ is infinite. Furthermore, it is shown in loc. cit. that there are examples of smooth projective surfaces $X$ over $p$-adic fields having good reduction for which $\CH^{2}(X) \lbrace p' \rbrace$ is not finite \cite[Theorem 1.1]{AS}. In \cite[Proposition 1.4]{Luder}, Lüders proved that for a smooth projective variety of dimension $d$ over a local field having a good reduction, the groups $_n\CH^{d+j}(X, j)$ are finite if $j \geq 1$ and $p\nmid n$. We generalize this result and prove the following.
       
       \begin{thm}\label{thm:Main-3}
       	       Let $X$ be a smooth projective variety of dimension $d  \geq 0$ over a local field $k$ that admits a good reduction. Then the $\CH^{i}(X, j) \lbrace p' \rbrace$ is finite if $j \geq 1$ and either $i \leq j+2$ or $i \geq d + 1$.         
        \end{thm}

        Note that Theorem~\ref{thm:Main-3} is a generalization of \cite[Proposition~1.4]{Luder} in two aspects: first, the above theorem establishes that the $p'$-torsion subgroup of $\text{CH}^{i}(X, j)$ is finite; second, the result covers a broader range of indices for $i$ and $j$. We also generalized this to smooth varieties over $N$-local fields and proved the following.

        \begin{thm}\label{thm:main:finiteness of p'-torsion in chow groups over N-fields}
 Let $X$ be a smooth projective variety of dimension $d \geq 0$ over $k_{N}$, where $N\geq 2$. Then  $\CH^{i}(X, j) \lbrace p' \rbrace$ is finite if $j \geq d+N$ and either $i \leq j+2$ or $i \geq d + N+1$. 
\end{thm}
See \cite{Hiranouchi-Sugiyama} for recent work on the groups $\CH^{i}(X, j)$, when $i = d+N+1$ and $j = N+1$, for smooth projective varieties $X$ over $k_{N}$.

\subsection{Application to Class field theory}
 Finally, we  give an application of these results to tame class field theory for smooth varieties over local fields. The tame reciprocity map $\rho^{t}_{X} : C^{t}(X) \rightarrow \pi^{ab,t}_{1}(X)$ for a smooth variety $X$ over a non-archimedian local field is well studied in \cite{GKR}. It is shown in loc. cit. that the kernel of the map $\rho^{t}_{X}$ is $p'$-divisible whenever  $X$ has a smooth compactification $\overline{X}$, which has a  good reduction. Here, we strengthen this result by showing that it is, in fact, uniquely $p'$-divisible. The analogous result for $\overline{X}$ over a $p$-adic field was shown in \cite[p. 185]{Szamuely}. We have the following:

      	\begin{thm}\label{thm-kernel-tame-p'-divisible}
      		The kernel of the tame reciprocity map $\rho^{t}_{X} : C^{t}(X) \rightarrow \pi^{ab,t}_{1}(X)$ is uniquely $p'$-divisible.
      	\end{thm}

\subsection{Outline of proofs}
       We briefly outline our proofs. The main ingredient to prove the theorems mentioned above is to prove the uniform finiteness of \'etale cohomology groups, i.e., to prove \thmref{thm:finiteness for quasiprojective over N-local fields}. We achieve this in Sections~\ref{sec:Et-Proj-local}, \ref{section:proj over local field} (for projective varieties) and \ref{section: quasi-proj over N-local} (for quasi-projective varieties).  The main input here is similar uniform boundedness results of
       Colliot-Th\'el\`ene-Sansuc-Soul\'e  \cite[Section~2.1]{CTSS} and Kahn \cite{Kahn-Jnt} over finite fields. 
       For a smooth  variety $X$ over a local field that admits a semi-stable reduction, using the above results, Gabber purity theorem and an induction, we prove the uniform finiteness for the \'etale cohomology groups of a semi-stable model of $X$ with support in the special fiber (see \thmref{support-sncsubvariety-codimension-finite}).

       When $X$ is also projective, we appeal to Gabber rigidity theorems and localization long exact sequence to conclude the proof of \thmref{thm:finiteness for quasiprojective over N-local fields} for such varieties (see \thmref{finiteness-smoothness-semistable model}). We then use alteration theorems of de Jong and Gabber
      and the flatification theorems of Raynaud-Gruson to conclude the proof of \thmref{thm:finiteness for quasiprojective over N-local fields} for all projective varieties over local fields (see \thmref{thm:Uni-bound-Proj-local}). In Section \ref{section:proj over local field}, we prove  similar results for projective varieties over higher local fields (see \propref{support-sncsubvariety-codimension-finite N-local field}, \thmref{key theorem}, and \thmref{projective-N-local field-bounded-restate} for the corresponding statements). To prove \thmref{thm:finiteness for quasiprojective over N-local fields} for a smooth quasi-projective schemes over an $N$-local field, the main idea is to use localization theorem, purity theorem and to approximate a singular scheme by regular schemes from inside (see Section \ref{section: quasi-proj over N-local} for details).

       In Section \ref{sec:Etale-N-local}, using \thmref{thm:finiteness for quasiprojective over N-local fields} and a  spectral sequence, we obtain divisible-by-finite type results for \'etale $K$-groups (see \thmref{thm:structure of etale-K-group over N-local}). The first part of \thmref{thm:Main:structure of algebraic-K-group over N-local} is then obtained in Section \ref{sec:Alg-K-N-local} using Quillen-Lichtenbaum Conjecture and results about the \'etale $K$-groups obtained in the previous section. To complete the proof of the theorem, we show that in a suitable range, $K$-groups of a smooth variety over a positive characteristic $N$-local field are $p$-divisible (see \lemref{lem:Gei-Lev-p-div}). Over perfect fields, this result is known by Geisser and Levine \cite{GL}. 
       We then discuss the various results (from previous sections) over finite fields and local fields in Section \ref{sec:Alg-K-groups-Finite-field} and obtain \thmref{thm:main:structure of etale-K-group over N-local for smooth} and Corollary \ref{cor:main:equivalence of conjectures}.

        One of the main ingredients for proving Theorem \ref{finiteness for chow groups over finite fields} and Theorem \ref{thm:Main-3} is to use Voevodsky comparison result to relate the motivic cohomology groups and the higher Chow groups. In Section \ref{sec:Mot-HCG}, we then use the results on {\et}ale realization map proved in \cite{GKR} and  
        \thmref{thm:finiteness for quasiprojective over N-local fields},  to conclude the proof of \thmref{finiteness for chow groups over finite fields}.

        In order to prove Theorem \ref{thm-kernel-tame-p'-divisible}, we compare the tame reciprocity map  with the {\et}ale realization map from compactly supported motivic cohomology to compactly supported {\et}ale cohomology groups with finite invertible coefficient. We already know that the kernel of this {\et}ale realization map is $p'$-divisible by our results from \cite{GKR}. It therefore suffices to show that the kernel does not have any prime-to-$p$ torsion. This is achieved in Section \ref{sec:App-CFT}.

\subsection{Notations}\label{sec:Notn}
We let $k$ be a field and let $\Sch_k$ denote the category of 
separated Noetherian $k$-schemes. We let $\Sm_k$ denote the category of smooth (in particular,
finite type) $k$-schemes. 
A $k$-scheme will mean an object of $\Sch_k$.
The product $X \times_{\Spec(k)} Y$ in $\Sch_k$ will be written as
$X \times Y$. We let $X^{(q)}$ (resp. $X_{(q)}$) denote the set of points on $X$ having
codimension (resp. dimension) $q$. We shall let $X_\sing$ (resp. $X_\reg$) denote the
singular locus (resp. regular locus) of $X$ with the reduced closed subscheme
structure. We let $\sZ_0(X)$ denote the free abelian group of 0-cycles on $X$.

We let $\Sch_{k/\zar}$ (resp. $\Sch_{k/\nis}$, resp. $\Sch_{k/ \etl})$
denote the Zariski (resp. Nisnevich, resp. {\'e}tale) site of $\Sch_{k}$.
Unless we mention the topology specifically,
all cohomology groups in this paper will be considered with respect to the
{\'e}tale topology.
We shall let $cd(X)$ denote the {\'e}tale cohomological
dimension of torsion sheaves on $X$.

For an abelian group $A$, we shall write $\Tor^1_{\Z}(A, {\Z}/n)$ as
$_nA$ and $A/{nA}$ as $A/n$. We shall let $A\{p'\}$ denote the
subgroup of elements of $A$ which are torsion of order prime to $p$.
We let $A\{p\}$ denote the subgroup of elements of $A$ which are torsion of 
order some power of $p$. For a prime $p $,  we say an abelian group $A$ is $p'$-divisible if $A/n = 0$ for all integers $n$ such that $(p, n) = 1$. We say $A$ is uniquely $p'$-divisible group if $A$ is $p'$-divisible and $A \lbrace p' \rbrace = 0$. The tensor product $A \otimes_{\Z} B$ will be written as
$A \otimes B$. For a set $\mathbb{M}$ of prime numbers, we shall let  $I_\M$ be the set of natural numbers whose prime divisors lie in  $\mathbb{M}$. We denote by $A_{\mathbb{M}}$, the 
inverse limit $\varprojlim_{m \in I_\M} \ A/{m}$, and shall call $A_\M$, the
$\M$-completion of $A$. 

\section{Finiteness in étale cohomology over finite and local fields} \label{sec:Et-Proj-local}

Throughout the section, we denote by $k$  any non-archimedian local field. We write  $H_{\et}^{i} (X, n) =\oplus_{l \neq p}  H_{\et}^{i} (X, \mathbb{Q}_{l}/\mathbb{Z}_{l}(n))$, where $l$ varies over all primes except $p$, the residue characteristic of $k$. In this section, we prove the following main result.

\begin{thm}\label{projective-local field-bounded}
    Let $X$ be a projective variety of dimension $d$ over a local field $k$. Then the group $H_{\et}^{i} (X, n)$  is finite if any one of the following holds.
    \begin{enumerate}
        \item  $i \nin [2n-d-1, n+d+1]$;
        \item $n \nin [0, d+1]$.
    \end{enumerate}
    \end{thm}
    
We have the following corollary. 

\begin{cor}
    Let $X$ be a projective variety of dimension $d \geq 0$ over a local field $k$. Then there exists an integer $M$ such that
    \[
	 \lvert H_{\et}^{i} (X, \mathbb{Z}/m(n) ) \rvert \leq M
	\]   for all integers $m$ such that $(m, p) = 1$ if any of the following holds.
    \begin{enumerate}
        \item $i \nin [2n-d-1, n+d+2];$
        \item $n \nin [0, d+1]$.
    \end{enumerate}
\end{cor}
\begin{proof}
    For any integer $m$ such that $(m, p) = 1$. We have the exact triangle
  \[
    {\Z}/m(n) \to ({\Q}/{\Z})'(n) \xrightarrow{m} ({\Q}/{\Z})'(n),
  \]
  where $({\Q}/{\Z})' = {\underset{\ell \neq p}\bigoplus} {\Q_\ell}/{\Z_\ell}$ and we have the following long exact sequence,
\begin{equation}
      \cdots \rightarrow H_{\et}^{i-1}(X, n) \rightarrow H_{\et}^{i}(X, \mathbb{Z}/m(n)) \rightarrow H_{\et}^{i}(X, n)  \rightarrow \cdots \end{equation}
Taking the limit over integers $m$ such that $(m,p) = 1$, the required result now follows from Theorem \ref{projective-local field-bounded}.
\end{proof}

We now prove Theorem \ref{projective-local field-bounded} and will achieve this in several steps. We first prove it when $X$ is smooth projective having semi-stable model and later reduce general case to the above case via alteration.  We begin with a recollection of known finiteness result for \'etale cohomology groups of varieties over finite fields.

\begin{thm}\label{finiteness-smooth-finite field}
	Let $X$ be a smooth variety of dimension $d$ over a finite field. Then $H_{\et}^{i} (X, n)$
	is finite if any of the following holds.
    \begin{enumerate}
        \item $i \nin [n, 2n+1]$.
        \item $i \nin [n, n+d+1]$.
        \item $n \nin [0, d]$.
        \item  $i \nin [2n, 2n+1]$ and $X$ is projective.
    \end{enumerate}
\end{thm}
\begin{proof}
    This follows from combining \cite[Theorem 1, Theorem 2]{Kahn-Jnt} and \cite[Theorem 2]{CTSS}.
\end{proof}

\begin{thm}\label{snc-projective-finite field}
    Let $X$ be a projective  variety over a finite field. Then the group $H_{\et}^{i} (X, n)$ is finite if $ i \nin [2n, n + d +1] $ or $n > d$.
\end{thm}

\begin{proof} This follows from \cite[Theorem 3(b)]{Kahn-Jnt}.
    \end{proof}

\begin{Def}
    Recall that a reduced Noetherian scheme $Y$ of pure dimension $d$ with
 irreducible components $\{Y_1, \ldots , Y_r\}$ is called a normal crossing scheme
 if for every nonempty subset subset $J \subset \{1, 2, ...,  r\}$, the scheme theoretic 
intersection $Y_J := {\underset{i \in J}\bigcap} Y_i$ is a regular scheme
which is either empty or of pure dimension $d + 1 - |J|$. It is clear from this
definition that if $Y$ is a normal crossing scheme, then
$Y_i \bigcap \ (\cup_{j \neq i} Y_j)$ is also a normal crossing scheme (of smaller
dimension than that of $Y$).
\end{Def}

\begin{Def}
    We let $\sX$ be a connected Noetherian regular scheme with a flat and projective
morphism $ \F \colon \sX \to S$ of relative dimension $d \le 1$, where $S$ is the spectrum of  a discrete valuation ring.
We let $X$ denote the generic
fiber and $\sX_s$ the (scheme-theoretic) closed fiber of $\F$. We shall assume that
$Y := (\sX_s)_\red$ is a simple normal crossing divisor on $\sX$. A morphism $\F$
satisfying these properties will be called a semi-stable model for the $k$-scheme $X$. Moreover, if $ \F $ is a smooth  morphism, we then say $X$ has a good reduction over $k$.
We let $u \colon X \inj \sX$ and $\iota \colon Y \inj \sX$ be the inclusions.
We let $f \colon X \to \Spec(k)$ and $g \colon \sX_s \to \Spec(\ff)$ denote the
structure maps.
\end{Def}

\begin{thm}\label{support-smoothsubvariety-codimension-finite}
	Let $X$ be a smooth variety of dimension $d \geq 0$ over a local field $k$ which has semistable model $\mathcal{X}$ with special fiber $A$. Suppose $Y$ is a smooth subvariety of codimension $c$ in $A$ over finite field. Then $H_{et, Y}^{i} (\mathcal{X}, n)$
	is finite if any of the following holds. 
    \begin{enumerate}
        \item $i \nin [n+c+1, n+d+2]$.
        \item $i \nin [n+c+1, 2n+1]$.
        \item $n \nin [c+1, d+1]$.
        \item  $i \nin [2n, 2n+1]$ and $\mathcal{X}$ is projective.
    \end{enumerate}
\end{thm}

\begin{proof}
   By Gabber's purity theorem \cite[Theorem 2.1.1]{F}, we have
\begin{equation}
H_{\et, Y}^{i} (\mathcal{X}, n) \cong H_{\et}^{i-2c-2} (Y, n-c-1). 
\end{equation}
 Since $Y$ is a smooth variety of dimension $d-c$ over a finite field, the result  follows from Theorem \ref{finiteness-smooth-finite field}.  
\end{proof}


\begin{cor}\label{good-reduction-finiteness}
	Let $X$ be a smooth projective variety of dimension $d \geq 0$ over local field $k$ which has good reduction $\mathcal{X}$ with special fiber $A$. Then $H_{\et}^{i} (X, n)$ is finite whenever $i \neq 2n-1, 2n, 2n+1$.
\end{cor}

\begin{proof}
    By the long exact sequence,
	\begin{equation}
	\cdots \rightarrow H_{\et}^{i}( \mathcal{X}, n) \rightarrow  H_{\et}^{i} (X,n) \rightarrow H_{\et, A}^{i+1}(\mathcal{X}, n) \rightarrow \cdots,
	\end{equation}
	it  suffices to show the finiteness of $H_{\et}^{i}( \mathcal{X}, n)$ and $H_{\et, A}^{i+1}(\mathcal{X}, n)$ in the given range of $i$.  Now by \cite[Chapter VI, Corollary 2.7]{Milne-etale}, we have $ H_{\et}^{i}( \mathcal{X}, n) \cong H_{\et}^{i}( A, n)$, which is finite by Theorem \ref{finiteness-smooth-finite field} if $i \neq 2n, 2n+1$. On the other hand, $H_{et, A}^{i+1}(\mathcal{X}, n)$ is finite by Theorem \ref{support-smoothsubvariety-codimension-finite} when  $i \neq 2n-1, 2n$. Combining both finiteness, we get the desired result.
\end{proof}


\begin{thm}\label{support-sncsubvariety-codimension-finite}
	Let $X$ be a smooth variety of dimension $d$ over a local field $k$ which has semistable model $\mathcal{X}$ with special fiber $A$. Suppose $Y$ a simple normal crossing subvariety of codimension $c$ in $A$ over the residue  field. Then $H_{\et, Y}^{i} (\mathcal{X}, n)$
	is finite if any of the following holds. 
    \begin{enumerate}
        \item $i \nin [n+c+1, n+d+2]$.
        \item $i \nin [n+c+1, 2n+1]$.
        \item $n \nin [c+1, d+1]$.
        \item  $i \nin [2n+c-d,  2n+1]$ and $\mathcal{X}$ is projective.
    \end{enumerate}
\end{thm}

\begin{proof}
We  prove this result by descending induction on $c$. First note that $c \leq d$.
Also if $c =d$, then $Y$ is smooth and hence we are done by Theorem \ref{support-smoothsubvariety-codimension-finite}. Assume that the result holds for all simple normal crossing subvarieties of codimension greater than $c$ inside $A$. We now prove the result for $Y$.  As $Y$ is simple normal crossing variety, we have the decomposition $Y = Y_{1} \cup \cdots \cup Y_{r}$ for some integer $r$, where each $Y_{i}$ is regular (hence smooth) variety over finite field for $1 \leq i \leq r$. Also, $Y$ is of pure codimension $c$ in $A$. We will prove this case by induction on $r$.	If $ r = 1$, then $Y$ is smooth and hence we are done by Theorem \ref{support-smoothsubvariety-codimension-finite}. Assume now that the result holds when number of irreducible components of $Y$ inside $A$ is less than $r$, where $r \geq 2$. Write $Y = Y_{1} \cup Y'_{1}$, where $Y'_{1} = Y_{2} \cup \cdots \cup Y_{r}$.
	
	Using the long exact sequence for support cohomology, we have the exact sequence 
	\begin{equation}
	H^{i}_{Y_{1}}(\mathcal{X}, n) \bigoplus H^{i}_{Y'_{1}}(\mathcal{X}, n)  \rightarrow H^{i}_{Y}(\mathcal{X}, n) \rightarrow H^{i+1}_{Y_{1} \cap Y'_{1}}(\mathcal{X}, n).
	\end{equation}
	Here, $H^{i}_{Y_{1}}(\mathcal{X}, n)$ is finite by the case $r = 1$, and the group $H^{i}_{Y'_{1}}(\mathcal{X}, n)$ is finite by the induction hypothesis on $r$ whenever $i$ and $n$ vary as in $(1)-(4)$.  Therefore we are reduced to show that $H^{i+1}_{Y_{1} \cap Y'_{1}}(\mathcal{X}, n)$ is finite in each of the above four cases, where  $Y_{1} \cap Y'_{1}$ is simple normal crossing scheme inside $A$ of codimension $c+1$.  We now use the induction hypothesis on $c$ to conclude that the group $H^{i+1}_{Y_{1} \cap Y'_{1}}(\mathcal{X}, n)$ is finite in each of the following cases:
     \begin{enumerate}
        \item[(1')] $i \nin [n+c+1, n+d+1]$.
        \item[(2')] $i \nin [n+c+1, 2n]$.
        \item[(3')] $n \nin [c+2, d+1]$.
        \item[(4')]  $i \nin [2n+ c-d, 2n]$ and $\mathcal{X}$ is projective.
    \end{enumerate}
    
Note that if the integers $i, n, d $ and $c$ satisfy one of the conditions $(1)-(4)$, then they also satisfy the corresponding condition from $(1')-(4')$. Hence, the group  $H^{i}_{Y}(\mathcal{X}, n)$ is finite in each of the given cases $(1)-(4)$. We have shown the result for $r$, which in particular also completes the proof for $c.$ This completes the proof.
\end{proof}

\begin{thm}\label{finiteness-smoothness-semistable model}
    Let $X$ be a smooth projective variety  of dimension $d$ over a local field $k$ which has a semistable model $\mathcal{X}$ with special fiber $A$. Then the group $H_{\et}^{i} (X, n)$  is finite if any of the following holds.
    \begin{enumerate}
        \item $i \nin [2n-d-1, n+d+1]$.
       \item $n \nin [0, d+1]$.
       \end{enumerate}
\end{thm}
\begin{proof}
 The part $(2)$ of this theorem is already known (see \cite[Lemma 10.9]{GKR}). Hence without loss of generality, one can assume that $n \leq d + 1$. In this case, we have $2n-d-1 \leq 2n \leq n+d+1$. We now show $(1)$. By the long exact sequence
	\begin{equation}
	\cdots \rightarrow H_{\et}^{i}( \mathcal{X}, n) \rightarrow  H_{\et}^{i} (X,n) \rightarrow H_{et, A}^{i+1}(\mathcal{X}, n) \rightarrow \cdots,
	\end{equation}
	it suffices to show the finiteness of $H_{\et}^{i}( \mathcal{X}, n)$ and $H_{et, A}^{i+1}(\mathcal{X}, n)$ in the given range of $i$.  Now by \cite[Chapter VI, Corollary 2.7]{Milne-etale}, we have $ H_{\et}^{i}( \mathcal{X}, n) \cong H_{\et}^{i}( A, n)$, which is finite by Theorem \ref{snc-projective-finite field} if $i \nin [2n, n+d+1]$. On the other hand, $H_{et, A}^{i+1}(\mathcal{X}, n)$ is finite by Theorem \ref{support-sncsubvariety-codimension-finite} when  $i \nin [2n-d-1, 2n]$. Combining both finiteness, we get the desired result.
	\end{proof}


\begin{thm} \label{thm:Uni-bound-Proj-local}
    Let $X$ be a projective variety of dimension $d\geq 0$ over a local field $k$. Then the group $H_{\et}^{i} (X, n)$  is finite if any one of the following holds.
    \begin{enumerate}
        \item  $i \nin [2n-d-1, n+d+1]$;
        \item $n \nin [0, d+1]$.
    \end{enumerate}  
\end{thm}

\begin{proof}
  The part $(2)$ is known (see \cite[Lemma 10.9]{GKR}). Hence without loss of generality, one can assume that $ 0 \leq  n \leq d + 1$. We  prove the part $(1)$ by the induction on $d$, the dimension of $X$. Since étale cohomology satisfy nil-invariance, we can assume that $X$ is reduced. Further, applying the Mayer-Vietoris sequence for closed covers, we can assume that $X$ is integral. In fact, this reduction can be achieved by similar methods as in Theorem \ref{snc-projective-finite field}.  
  If $d = 0$, we are done by Theorem \ref{finiteness-smoothness-semistable model}. So we assume that $d \geq 1$.  By \cite[Theorem 4.3.1]{Temkin} (also see \cite[Lemma 10.9]{GKR}), we can find a finite field extension ${k'}/k$
      and a smooth projective integral $k'$-scheme $W$ together with a projective
      morphism $\pi \colon W \to X$ such that the following hold.
      \begin{enumerate}
      \item[(1')]
        $k' = k$ and $\pi$ is birational if $\Char(k) = 0$.
      \item[(2')]
        The degrees  $[k':k]$ is some power of $p$
        if $\Char(k) = \Char(\ff) = p$.
      \item[(3')] The morphism 
        $\pi$ is a $p$-alteration (an alteration {\`a} la de Jong whose degree is a
        $p$-power) if $\Char(k) = p$.
      \item[(4')]
        $W \to \Spec(k')$ admits a semistable reduction.
      \end{enumerate}
 If $\Char(k) = 0$, then we can find a resolution of singularities
  $\pi \colon W \to X$. If we let $Z = X_\sing$ and $E = \pi^{-1}(Z)$, then
  $\{W \amalg Z \rightarrow X\}$ is a cdh cover of $X$ such that
  $\dim(Z) < \dim(X)$ and $\dim(E) < \dim(X)$.
 A combination of proper base change theorem and
localization theorem in {\'e}tale cohomology
(see \cite[Rmk. III.3.13, Cor. VI.2.3]{Milne-etale}) implies that we have an exact sequence (see \cite[Lemma 10.2]{GKR})
\begin{equation}\label{eqn:Real-iso-3}
    \xymatrix@C.8pc{
{\begin{array}{l}
        H_{\et}^{i-1}(W, n) \\
       \hspace*{1.2cm} \oplus \\
       H_{\et}^{i-1}(Z, n)
     \end{array}}
   \ar[r] &   H_{\et}^{i-1}(E, n) \ar[r] & H_{\et}^i(X, n)
\ar[r] & {\begin{array}{l}
        H_{\et}^{i}(W, n) \\
       \hspace*{1.2cm} \oplus \\
       H_{\et}^{i}(Z, n)
     \end{array}}
   \ar[r] &  H_{\et}^{i}(E, n).}
\end{equation}
  Since $\dim(Z) < \dim(X)$,  the induction hypothesis implies that the group $H_{\et}^{i}(Z, n)$ is finite whenever $i \nin [2n-\text{dim}(Z)-1, n+\text{dim}(Z)+1]$. Similarly, since $\dim(E) < \dim(X)$, the group $H_{\et}^{i-1}(E, n)$ is also finite if  $ i-1 \nin [2n-\text{dim}(E)-1, n+ $dim$(E)+1] $, which holds iff $ i \nin [2n-\text{dim}(E), n+ $dim$(E)+2] $. These conditions are clearly satisfied whenever $ i \nin [2n-d-1, n+ d+1] $ as dim$(Z) \leq d-1$, and dim$(E) \leq d-1$. Hence, we are reduced to show the group $ H_{\et}^{i}(W, n)$ is finite whenever $ i \nin [2n-d-1 , n+d+1]$. But this follows from Theorem \ref{finiteness-smoothness-semistable model}.

 We now assume that $\Char(k) = p > 0$.
Using Temkin's strengthening of the alteration theorems of de Jong and Gabber
(see \cite[Thm.~1.2.5]{Temkin}) and the flatification theorems of Raynaud-Gruson, it is shown in the proof of
\cite[Lemma 10.2]{GKR} that there exists a commutative diagram
\begin{equation}\label{eqn:Real-iso-1}
  \xymatrix@C1pc{
    W' \ar[r]^-{\pi'} \ar[d]_-{f'} & X' \ar[d]^-{f} \\
    W \ar[r]^-{\pi} & X,}
\end{equation}
where $f$ is the blow-up along a nowhere dense closed subscheme $Z \inj X$ and $f'$ is
the strict transform of $W$ along $f$. The map $\pi$ is a $p$-alteration of $X$
such that $W$ is integral and smooth over $k$ and $\pi'$ is a finite and flat
morphism between integral schemes of degree $p^r$ for some $r \ge 0$.

Since $\{X' \amalg Z \rightarrow X\}$ is a cdh cover of $X$, we can apply the above
characteristic zero
argument to reduce the proof to showing that the lemma holds for
${X'}$.
By applying this argument to $f'$ and using induction,
we see that the lemma holds for $W'$ because it
holds for $W$ by Theorem \ref{finiteness-smoothness-semistable model}.

The assertion for $X'$ now follows as the map $\pi'$ is finite flat map of degree $p^r$, hence by pull-back push-forward  arguments for etale cohomology (see \cite[Lemma 10.2]{GKR}), we have a commutative
diagram
\begin{equation}\label{eqn:Real-iso-2.0}
  \xymatrix@C1pc{
    H^i(X', n) \ar[r]^-{\pi'^*} \ar[d]_-{\epsilon^*_{X'}} &
    H^i(W', n) \ar[r]^-{\pi'_*} \ar[d]^-{\epsilon^*_{W'}} &
    H^i(X', n) \ar[d]^-{\epsilon^*_{X'}} \\
    H^i_\et(X', n) \ar[r]^-{\pi'^*} &
    H^i_\et(W', n) \ar[r]^-{\pi'_*}  &
    H^i_\et(X', n)}
\end{equation}
such that the composition of the horizontal arrows on each of the two rows is
multiplication by $p^r$. This implies that $H^i_\et(X', n) \hookrightarrow H^i_\et(W', n)$ as $(n,p) = 1$. This concludes the proof.
\end{proof}

\section{Finiteness in étale cohomology  over higher local fields}\label{section:proj over local field}

Let $N$ be any non-negative integer. 
Recall that an $N$-local field is defined inductively so that  it is complete discrete valued field $k_N$ such that its residue field $k_{N-1}$ is an $(N-1)$-local field, where $0$-local field is a finite field. 
In particular, $k_N$ is an $N$-local field if and only if there exists an $(N+1)$-tuple  $(k_N, k_{N-1}, \dots, k_1, k_0)$ of fields such that each $k_i$ is a complete discrete valuation field, $k_{i-1}$ is the residue field of $k_i$, and $k_0$ is a finite field.
We let $p>0$ denote the  characteristic of $k_0$. 
We call these fields as higher local fields. With this definition,  finite fields are $0$-local fields and local fields are $1$-local fields. These fields were introduced by K. Kato. In this section, we generalize the finiteness result proved in the previous section over local fields to $N$-local fields for any $N\geq 1$. For any quasi-projective $k_{N}$-variety $X$, we write  $H_{\et}^{i} (X, n) =\oplus_{l \neq p}  H_{\et}^{i} (X, \mathbb{Q}_{l}/\mathbb{Z}_{l}(n))$, where $l$ varies over all primes except $p$, the characteristic of the finite field $k_{0}$. The main result of this section is the following:

\begin{thm}\label{projective-N-local field-bounded}
    Let $X$ be a projective variety  of dimension $d \geq 0$ over an $N$-local field $k_{N}$. Then the group $H_{\et}^{i} (X, n)$  is finite if any one of the following holds.
    \begin{enumerate}
        \item  $i \nin [2n-d-N, n+d+1]$;
        \item $n \nin [0, d+N]$.
    \end{enumerate}
    \end{thm}

Using this, one can deduce the following.

\begin{cor}
    Let $X$ be a projective variety of dimension $d$ over $k_{N}$. Then there exists an integer $M$ such that
\[
	 \lvert H_{\et}^{i} (X, \mathbb{Z}/l^{v}(n) ) \rvert \leq M
	\]  for all $v \geq 0$ and $l \neq p$ if any of the following  holds.
    \begin{enumerate}
        \item $i \nin [2n-d-N, n+d+2];$
        \item $n \nin [0, d+N]$.
    \end{enumerate}
\end{cor}
\begin{proof}
    For any integer $m$ such that $(m, p) = 1$. We have the exact triangle
  \[
    {\Z}/m(n) \to ({\Q}/{\Z})'(n) \xrightarrow{m} ({\Q}/{\Z})'(n),
  \]
  where $({\Q}/{\Z})' = {\underset{\ell \neq p}\bigoplus} {\Q_\ell}/{\Z_\ell}$ and we have the following long exact sequence,
\begin{equation}
      ... \rightarrow H_{\et}^{i-1}(X, n) \rightarrow H_{\et}^{i}(X, \mathbb{Z}/m(n)) \rightarrow H_{\et}^{i}(X, n)  \rightarrow... \end{equation}
Taking the limit over integers $m$ such that $(m,p) = 1$, the required result now follows from Theorem \ref{projective-N-local field-bounded}.
\end{proof}

We prove the theorem in several steps. We begin with the finiteness result for $\Spec(k_{N})$.

\begin{prop}\label{finiteness-N-local field}
    If  $X = \Spec(k_{N})$, then $H_{\et}^{i} (\Spec(k_{N}), n)$
	is finite if any of the following holds. 
    \begin{enumerate}
        \item $i \nin [2n-N, n+1]$.
       \item $n \nin [0, N]$.
       \end{enumerate}
\end{prop}
\begin{proof}
 First note that when $N = 0$, this follows from Theorem \ref{finiteness-smooth-finite field}. In this case, the result says that $H_{\et}^{i} (\Spec(k_{0}), n)$ is finite either $n \neq 0$ or $i \nin [0,1]$.  So we assume that $N \geq 1$, and prove this result by induction on $N$. As $X = \Spec(k_{N})$, we can choose $\mathcal{X} = \Spec(\mathcal{O}_{{k}_{N}}) $ and $A = \Spec(k_{N-1})$. 
 Assume that $N \geq 1$ and the results holds for $N-1$.  Using the long exact sequence, we have
\begin{equation}
	... \rightarrow H_{\et}^{i}( \mathcal{X}, n) \rightarrow  H_{\et}^{i} (X,n) \rightarrow H_{\et, A}^{i+1}(\mathcal{X}, n) \rightarrow ..,
	\end{equation}
it suffices to show the finiteness of $H_{ \et}^{i}( \mathcal{X}, n)$ and $H_{et, A}^{i+1}(\mathcal{X}, n)$ in the given range of $i$ and $n$.  Now by \cite[Chapter VI, Corollary 2.7]{Milne-etale}, we have $ H_{\et}^{i}( \mathcal{X}, n) \cong H_{\et}^{i}( A, n)$.  The induction hypothesis implies that $H_{\et}^{i}( A, n)$ is finite if either $i \nin [2n-(N-1), n+1]$ or $n \nin [0, N-1]$. In particular, these conditions are satisfied for given ranges of $i$ and $n$. On the other hand, by purity isomorphism, we have $H_{\et, A}^{i+1}(\mathcal{X}, n) \cong H_{\et}^{i-1}( A, n-1)$. Applying the induction hypothesis, $H_{\et}^{i-1}( A, n-1)$ is finite if either  $i-1 \nin [2(n-1)-N+1, n-1 + 1]$ or $n \nin [0,  N-1]$, which holds iff either  $i \nin [2n-N, n+1]$ or $n \nin [0,  N-1]$. In particular, these conditions are also satisfied in the given range of $i$ and $n$. This completes the proof.
\end{proof}

\begin{prop}\label{support-regular-subvariety-codimension-finite N-local field}
	Let $X$ be a smooth projective variety of dimension $d$ over $k_{N}$ which has semi-stable model $\mathcal{X}$ with special fiber $A$. Suppose $Y$ is regular subvariety of codimension $c$ in $A$ over $k_{N-1}$. Assume that  Theorem \ref{projective-N-local field-bounded} holds for all projective varieties over $k_{N-1}$. Then $H_{\et, Y}^{i+1} (\mathcal{X}, n)$
	is finite if any of the following holds. 
    \begin{enumerate}
        \item $i \nin [2n+c-d-N, n+d+1]$.
       \item $n \nin [c+1, d+N]$.
       \end{enumerate}
    
\end{prop}

\begin{proof}
      By Gabber's purity theorem \cite[Theorem 2.1.1]{F}, we have
\begin{equation}
H_{\et, Y}^{i+1} (\mathcal{X}, n) \cong H_{\et}^{i+1-2(c+1)} (Y, n-c-1) \cong H_{\et}^{i-2c-1} (Y, n-c-1) . 
\end{equation}
Since $Y$ is a projective variety over $k_{N-1}$,  our assumption implies that  $H_{\et}^{i-2c-1} (Y, n-c-1)$ is finite if any of the following holds;
\begin{enumerate}
    \item[$(1')$] $i-2c-1 \nin [2(n-c-1)-(d-c)-(N-1), n-c-1+ d-c+1]$;
    \item[$(2')$] $n-c-1 \nin [0, d-c+(N-1)]$.
\end{enumerate}
Here, the condition $(1')$ is equivalent to $i \nin [2n-(d-c)-N, n+d+1]$. 
Also, $(2')$ is equivalent to $n \nin [c+1, d+N]$, which 
completes the proof.
\end{proof}


\begin{prop}\label{support-sncsubvariety-codimension-finite N-local field}
	Let $X$ be a smooth projective variety of dimension $d$ over $k_{N}$ which has semi-stable model $\mathcal{X}$ with special fiber $A$. Suppose $Y$ a simple normal crossing subvariety of codimension $c$ in $A$ over $k_{N-1}$. Assume that Theorem \ref{projective-N-local field-bounded} holds for all projective varieties over $k_{N-1}$. Then $H_{\et, Y}^{i+1} (\mathcal{X}, n)$
	is finite if any of the following holds. 
    \begin{enumerate}
        \item $i \nin [2n+c-d-N, n+d+1]$.
       \item $n \nin [c+1, d+N]$.
       \end{enumerate}
    
\end{prop}

\begin{proof}
We will prove this result by descending induction on $c$. First note that $c \leq d$. Also if $c =d$, then $Y$ is regular and hence we are done by Proposition \ref{support-regular-subvariety-codimension-finite N-local field}. Assume that the result holds for all simple normal crossing subvarieties of codimension greater than $c$ inside $A$. We now prove the result for $Y$.  As $Y$ is simple normal crossing variety, so $Y = Y_{1} \cup ... \cup Y_{r}$ for some integer $r$, where each $Y_{i}$ is regular variety over $k_{N-1}$ for $1 \leq i \leq r$. Also, $Y$ is of pure codimension $c$ in $A$. We will prove this case by induction on $r$.	If $ r = 1$, then $Y$ is regular and hence we are again done by Proposition \ref{support-regular-subvariety-codimension-finite N-local field}. So assume results holds when number of irreducible components of $Y$ inside $A$ is less than $r$, where $r \geq 2$. Write $Y = Y_{1} \cup Y'_{1}$, where $Y'_{1} = Y_{2} \cup ... \cup Y_{r}$. Using the long exact sequence for support cohomology, we have
\begin{equation}\label{support-long-sequence}
	H^{i+1}_{Y_{1}}(\mathcal{X}, n) \bigoplus H^{i+1}_{Y'_{1}}(\mathcal{X}, n)  \rightarrow H^{i+1}_{Y}(\mathcal{X}, n) \rightarrow H^{i+2}_{Y_{1} \cap Y'_{1}}(\mathcal{X}, n).
	\end{equation}
	
Here, $H^{i+1}_{Y_{1}}(\mathcal{X}, n)$ is finite by the case $r = 1$, and the group $H^{i+1}_{Y'_{1}}(\mathcal{X}, n)$ is finite by induction hypothesis on $r$ in the given ranges of $i$ and $n$.  Therefore we are reduced to show that $H^{i+2}_{Y_{1} \cap Y'_{1}}(\mathcal{X}, n)$ is finite in the given ranges of $i$ and $n$. Note that $Y_{1} \cap Y'_{1}$ is simple normal crossing scheme inside $A$ of codimension $c+1$.  We now use the induction hypothesis on $c$ to conclude that the group $H^{i+2}_{Y_{1} \cap Y'_{1}}(\mathcal{X}, n)$ is finite whenever $i$ and $n$ satisfy $(1)$ or $(2)$.  Hence, $(\ref{support-long-sequence})$ implies that the group $H^{i}_{Y}(\mathcal{X}, n)$ is finite when $i$ and $n$ satisfy $(1)$ or $(2)$. This completes the proof by induction for $r$ and in particular completes the proof by induction on $c$. This completes the proof of the proposition.
\end{proof}

  

\begin{thm}\label{key theorem}
    Let $N \geq 2$ be a fixed integer and $X$ be a projective variety of dimension $d$ over $k_{N}$. Assume that Theorem $\ref{projective-N-local field-bounded}$ holds for all projective varieties over $k_{N-1}$. Then $H_{\et}^{i} (X, n)$
	is finite if any of the following holds.
    \begin{enumerate}
        \item $i \nin [2n-d-N, n+d+1]$.
       \item $n \nin [0, d+N]$.
       \end{enumerate}
    
\end{thm}
\begin{proof}

We shall prove this using the induction on $d$, the dimension of $X$. Since étale cohomology groups satisfy nil-invariance, we can assume that $X$ is reduced and further applying the Mayer-Vietoris sequence for closed covers, we can assume that $X$ is integral. In fact, this reduction can be achieved by similar methods as in Theorem \ref{snc-projective-finite field}.  
  If $d = 0$, we are done by Proposition \ref{finiteness-N-local field}. So we shall assume that $d \geq 1$.  By \cite[Theorem 4.3.1]{Temkin}(also see \cite[Lemma 10.9]{GKR}), we can find a finite field extension ${k_{N}'}/k_{N}$
      and a smooth projective integral $k_{N}'$-scheme $W$ together with a projective
      morphism $\pi \colon W \to X$ such that the following hold.
      \begin{enumerate}
      \item[(1')]
        $k_{N}' = k_{N}$ and $\pi$ is birational if $\Char(k_{N}) = 0$.
      \item[(2')]
        The degrees  $[k_{N}':k_{N}]$ is some power of $p$
        if $\Char(k_{N}) = p$.
      \item[(3')] The morphism 
        $\pi$ is a $p$-alteration (an alteration {\`a} la de Jong whose degree is a
        $p$-power) if $\Char(k_{N}) = p$.
      \item[(4')]
        $W \to \Spec(k')$ admits a semistable reduction.
      \end{enumerate}
 If $\Char(k_{N}) = 0$, then we can find a resolution of singularities
  $\pi \colon W \to X$. If we let $Z = X_\sing$ and $E = \pi^{-1}(Z)$, then
  $\{W \amalg Z \rightarrow X\}$ is a cdh cover of $X$ such that
  $\dim(Z) < \dim(X)$ and $\dim(E) < \dim(X)$.
 A combination of proper base change theorem and
localization theorem in {\'e}tale cohomology
(see \cite[Rmk. III.3.13, Cor. VI.2.3]{Milne-etale}) implies that we have an exact sequence (see \cite[Lemma 10.2]{GKR})
\begin{equation}\label{eqn:Real-iso-3.1}
    \xymatrix@C.8pc{
{\begin{array}{l}
        H_{\et}^{i-1}(W, n) \\
       \hspace*{1.2cm} \oplus \\
       H_{\et}^{i-1}(Z, n)
     \end{array}}
   \ar[r] &   H_{\et}^{i-1}(E, n) \ar[r] & H_{\et}^i(X, n)
\ar[r] & {\begin{array}{l}
        H_{\et}^{i}(W, n) \\
       \hspace*{1.2cm} \oplus \\
       H_{\et}^{i}(Z, n)
     \end{array}}
   \ar[r] &  H_{\et}^{i}(E, n).}
\end{equation}
  Since $\dim(Z) < \dim(X)$, the induction hypothesis implies that  the group $H_{\et}^{i}(Z, n)$ is finite whenever $i \nin [2n-\text{dim}(Z)-N, n+\text{dim}(Z)+1]$. Similarly, since $\dim(E) < \dim(X)$, $H_{\et}^{i-1}(E, n)$ is also finite if  $ i-1 \nin [2n-\text{dim}(E)-N, n+ $dim$(E)+1] $, which holds iff $ i \nin [2n-\text{dim}(E)-N+1, n+ $dim$(E)+2] $. These conditions are clearly satisfied whenever $ i \nin [2n-d-N, n+ d+1] $ as dim$(E) \leq d-1$. Hence we are reduced to show the group $ H_{\et}^{i}(W, n)$ is finite whenever $ i \nin [2n-d-N , n+d+1]$. This follows from Proposition \ref{support-sncsubvariety-codimension-finite N-local field}.

 We now assume that $\Char(k) = p > 0$.
Using Temkin's strengthening of the alteration theorems of de Jong and Gabber
(see \cite[Thm.~1.2.5]{Temkin}) and the flatification theorems of Raynaud-Gruson, it is shown in the proof of
\cite[Lemma 10.2]{GKR} that there exists a commutative diagram
\begin{equation}\label{eqn:Real-iso-1.1}
  \xymatrix@C1pc{
    W' \ar[r]^-{\pi'} \ar[d]_-{f'} & X' \ar[d]^-{f} \\
    W \ar[r]^-{\pi} & X,}
\end{equation}
where $f$ is the blow-up along a nowhere dense closed subscheme $Z \inj X$ and $f'$ is
the strict transform of $W$ along $f$. The map $\pi$ is a $p$-alteration of $X$
such that $W$ is integral and smooth over $k$ and $\pi'$ is a finite and flat
morphism between integral schemes of degree $p^r$ for some $r \ge 0$.

Since $\{X' \amalg Z \rightarrow X\}$ is a cdh cover of $X$, we can apply the above
characteristic zero
argument to reduce the proof to showing that the lemma holds for
${X'}$.
By applying this argument to $f'$ and using induction,
we see that the lemma holds for $W'$ because it
holds for $W$ by Proposition \ref{support-sncsubvariety-codimension-finite N-local field}.

The assertion for $X'$ now follows as the map $\pi'$ is finite flat map of degree $p^r$, hence by pull-back push-forward  arguments for étale cohomology (see \cite[Lemma 10.2]{GKR}), we have a commutative
diagram
\begin{equation}\label{eqn:Real-iso-2}
  \xymatrix@C1pc{
    H^i(X', n) \ar[r]^-{\pi'^*} \ar[d]_-{\epsilon^*_{X'}} &
    H^i(W', n) \ar[r]^-{\pi'_*} \ar[d]^-{\epsilon^*_{W'}} &
    H^i(X', n) \ar[d]^-{\epsilon^*_{X'}} \\
    H^i_\et(X', n) \ar[r]^-{\pi'^*} &
    H^i_\et(W', n) \ar[r]^-{\pi'_*}  &
    H^i_\et(X', n)}
\end{equation}
such that the composition of the horizontal arrows on each of the two rows is
multiplication by $p^r$. This implies that $H^i_\et(X', n) \hookrightarrow H^i_\et(W', n)$ as $(n,p) = 1$. This concludes the proof.
\end{proof}

 We now begin proving the main result of this section. We recall the statement for reader's convenience.

\begin{thm}\label{projective-N-local field-bounded-restate}
    Let $N \geq 1$ be an integer and let $X$ be a projective variety of dimension $d$ over an $N$-local field $k_{N}$. Then the group $H_{\et}^{i} (X, n)$  is finite if any one of the following holds.
    \begin{enumerate}
        \item  $i \nin [2n-d-N, n+d+1]$;
        \item $n \nin [0, d+N]$.
    \end{enumerate}
    \end{thm}

\begin{proof}
This follows from Theorem \ref{key theorem} via an easy induction on $N$ as the desired result holds for $N = 1$ by Theorem \ref{projective-local field-bounded}.  This completes the proof.
\end{proof}

\begin{cor}\label{quasiprojective-N-local field-bounded}
    Let $N \geq 1$ be an integer and $X$ be a projective variety of dimension $d$ over an $N$-local field $k_{N}$. Suppose $U$ be a non-empty open subset of $X$. Then the group $H_{c, \et}^{i} (U, n)$  is finite if any one of the following holds.
    \begin{enumerate}
        \item  $i \nin [2n-d-N, n+d+1]$;
        \item $n \nin [0, d+N]$.
    \end{enumerate}
    \end{cor}
\begin{proof}
We have the long exact sequence for compactly supported étale cohomology groups,
\begin{center}
    $ ... \rightarrow H_{c, \et}^{i-1} (X \setminus U, n) \rightarrow H_{c, \et}^{i} (U, n) \rightarrow H_{c, \et}^{i} (X, n) \rightarrow ... $ 
\end{center}
Since $X \setminus U$ and $X$ are projective, we have $H_{c, \et}^{i-1} (X \setminus U, n) \cong H_{\et}^{i-1} (X \setminus U, n)$, and $H_{c, \et}^{i} (X, n) \cong H_{\et}^{i} (X, n)$. Both these groups are finite by Theorem \ref{projective-N-local field-bounded-restate} for given ranges of $i$ and $n$. This completes the proof.
\end{proof}

\section{Finiteness for quasi-projective varieties over $N$-local fields}
\label{section: quasi-proj over N-local}

We keep the notations same as in the beginning of \S~3. In this section, we generalize the finiteness results proved for projective varieties in the previous section to quasi-projective varieties over  $N$-local fields for any integer $N\geq 1$. For any quasi-projective $k_{N}$-variety $X$, we write  $H_{\et}^{i} (X, n) =\oplus_{l \neq p}  H_{\et}^{i} (X, \mathbb{Q}_{l}/\mathbb{Z}_{l}(n))$, where $l$ varies over all primes except $p$, the characteristic of finite field $k_{0}$. The main result of this section is the following:

\begin{thm}\label{quasi-projective-N-local field-bounded}
    Let $X$ be a quasi-projective variety of dimension $d$ over an $N$-local field $k_{N}$. Then the group $H_{\et}^{i}(X, n)$  is finite if any one of the following holds.
    \begin{enumerate}
        \item  $i \nin [2n-d-N, n+d+1]$;
        \item $n \nin [0, d+N]$.
    \end{enumerate}
    \end{thm}

Using this, one can deduce the following.

\begin{cor}
    Let $X$ be a quasi-projective variety of dimension $d$ over $k_{N}$. Then there exists an integer $M$ such that
\[
	 \lvert H_{\et}^{i} (X, \mathbb{Z}/m(n) ) \rvert \leq M
	\]  for all integers $m$ such that $(m, p) = 1$ if any of the following holds.
    \begin{enumerate}
        \item $i \nin [2n-d-N, n+d+2];$
        \item $n \nin [0, d+N]$.
    \end{enumerate}
\end{cor}
\begin{proof}
    For any integer $m$ such that $(m, p) = 1$. We have the exact triangle
  \[
    {\Z}/m(n) \to ({\Q}/{\Z})'(n) \xrightarrow{m} ({\Q}/{\Z})'(n),
  \]
  where $({\Q}/{\Z})' = {\underset{\ell \neq p}\bigoplus} {\Q_\ell}/{\Z_\ell}$ and we have the following long exact sequence
\begin{equation}
      \cdots \rightarrow H_{\et}^{i-1}(X, n) \rightarrow H_{\et}^{i}(X, \mathbb{Z}/m(n)) \rightarrow H_{\et}^{i}(X, n)  \rightarrow \cdots \end{equation}
Taking the limit over integers $m$ such that $(m,p) = 1$, the required result now follows from Theorem \ref{projective-N-local field-bounded}.
\end{proof}

\subsection{Finiteness for curves} In this subsection, we prove the main results of this section for curves.

\begin{prop}\label{prop:finiteness can be shrinked-curve}
    Let $N \geq 1$ be any integer. Suppose $X'$ is a regular quasi-projective integral curve over $k_{N}$, and $X \subset X'$ is a non-empty  dense open subset of $X'$. Then the natural map
    \begin{center}
         $H_{\et}^{i} (X', n) \rightarrow H_{\et}^{i} (X, n) $
    \end{center}
    has finite kernel and cokernel if any one of the following holds;
    \begin{enumerate}
        \item  $i \nin [2n-N-1, n+2]$;
        \item $n \nin [0, N+1]$.
    \end{enumerate}

    
    \end{prop}

\begin{proof}
Set $X'-X = Z$, which is a closed subscheme of $X'$. Since $X'$ is a curve, we have $Z =  \bigcup_{\alpha \in I_{1}} Z^{\alpha}$, where $Z^{\alpha}$ are $0$-dimensional schemes and  $I_{1}$ is a finite set. We have the following long exact sequence of support cohomology groups.
\begin{equation}\label{support-sequence for X, X'-curve}
   H_{Z}^{i}(X', n) \rightarrow H_{\et}^{i}(X', n) \rightarrow H_{\et}^{i}(X, n) \rightarrow H_{Z}^{i+1}(X', n)     
   \end{equation}
By Gabber purity isomorphism, we have
\begin{equation}
    H_{Z}^{i}(X', n) \cong \oplus_{\alpha \in I_{1}} H_{Z^{\alpha}}^{i}(X', n) \cong \oplus H_{\et}^{i-2}(Z^{\alpha}, n-1), \hspace{2mm} \text{and}
\end{equation}
\begin{equation}
     H_{Z}^{i+1}(X', n) \cong \oplus_{\alpha \in I_{1}} H_{Z^{\alpha}}^{i+1}(X', n) \cong \oplus H_{\et}^{i-1}(Z^{\alpha}, n-1).
\end{equation}

Using Proposition \ref{finiteness-N-local field}, $H_{\et}^{i-2}(Z^{\alpha}, n-1)$ is finite if either $i-2 \nin [2(n-1)-N, (n-1)+1] $ or $n-1 \nin [0, N]$, which is equivalent to either $i \nin [2n-N, n+2] $ or $n \nin [1, N+1]$. These conditions are satisfied by $(1)$ and $(2)$.  Also,  $H_{\et}^{i-1}(Z^{\alpha}, n-1)$ is finite if either $i-1 \nin [2(n-1)-N, (n-1)+1] $ or $n-1 \nin [0, N]$, which is equivalent to either $i \nin [2n-N-1, n+1] $ or $n \nin [1, N+1]$. These conditions are satisfied whenever $i$ and $n$ satisfy $(1)$ and $(2)$. This completes the proof in view of \eqref{support-sequence for X, X'-curve}.
\end{proof}



\begin{thm}\label{quasi-projective-bounded-curve} Let $X$ be a quasi-projective curve over an $N$-local field $k_{N}.$ Then $H_{\et}^{i} (X, n)$ is finite if any of the following holds.
\begin{enumerate}
        \item  $i \nin [2n-N-1, n+2]$;
        \item $n \nin [0, N+1]$.
    \end{enumerate}

\end{thm}
\begin{proof}
 Without loss of generality, one can assume that $X$ is integral. Consider the normalization map
  $\pi \colon W \to X$. If we let $Z = X_\sing$ and $E = \pi^{-1}(Z)$, then
  $\{W \amalg Z \rightarrow X\}$ is a cdh cover of $X$ such that
  $\dim(Z) = 0$ and $\dim(E) = 0$. We have the following long exact sequence.
 \begin{equation}\label{eqn:Real-iso-3.2}
    \xymatrix@C.8pc{
{\begin{array}{l}
        H_{\et}^{i-1}(W, n) \\
       \hspace*{1cm} \oplus \\
       H_{\et}^{i-1}(Z, n)
     \end{array}}
   \ar[r] &   H_{\et}^{i-1}(E, n) \ar[r] & H_{\et}^i(X, n)
\ar[r] & {\begin{array}{l}
        H_{\et}^{i}(W, n) \\
       \hspace*{1cm} \oplus \\
       H_{\et}^{i}(Z, n)
     \end{array}}
   \ar[r] &  H_{\et}^{i}(E, n)}.
\end{equation}
The groups $H_{\et}^{i-1}(E, n)$ and $H_{\et}^{i}(Z, n)$ are finite in the given ranges of $i$ and $n$ by Proposition \ref{finiteness-N-local field}.   We are therefore reduced to show that $H_{\et}^{i}(W, n)$ is finite in the given range of $i$ and $n$, where $W$ is regular quasi-projective curve over $k_{N}$. 
We can find a regular projective integral curve $W'$ such that $W$ is a dense open subset of $W'$. Therefore by Proposition \ref{prop:finiteness can be shrinked-curve}, we are reduced to show similar finiteness for $W'$, which is shown in Theorem \ref{projective-N-local field-bounded-restate}. This completes the proof.
\end{proof}

\subsection{Finiteness in higher dimensions} In this subsection, we prove the finiteness results for quasi-projective varieties of arbitrary dimensions.

\begin{lem}\label{finiteness can be shrinked for higher dimension}
    Let $N \geq 1$ be any integer. Suppose $X'$ is a regular quasi-projective variety of dimension $d \geq 1$ over $k_{N}$, and let $X \subset X'$ be a non-empty open dense subset of $X'$. Assume that Theorem \ref{quasi-projective-N-local field-bounded} holds for all quasi-projective varieties over $k_{N}$ of dimension less than $d$. Then the natural map
    \begin{center}
        $H_{\et}^{i} (X', n) \rightarrow H_{\et}^{i} (X, n)$
    \end{center}
   has finite kernel and cokernel if any one of the following holds.
    \begin{enumerate}
        \item  $i \nin [2n-N-d, n+d+1]$;
        \item $n \nin [0, d+N]$.
    \end{enumerate}
    \end{lem}

\begin{proof}
This Lemma holds for $d =1$ by Proposition \ref{prop:finiteness can be shrinked-curve}, and hence we can assume that $d \geq 2$.  We set $Z = X' - X$, which is a closed subscheme of $X'$ with  dim($Z) < $ dim$(X')$. Also, set $Z_{1} = Z_{sing}$, the singular locus of $Z$, and  $X_{1} = X' - Z_{1} $. Note that  $X_{1} - X  = Z - Z_{1} $, which is a regular closed subscheme of $X_{1}$. So we can write, $ Z - Z_{1}  = \bigcup_{\alpha \in I_{1}} Z_{1}^{\alpha}$, where $Z^{\alpha}_{1}$ is regular connected scheme of codimension $c_{\alpha}$ in $X_{1}$, and $I_{1}$ is a finite set.   We then have a long exact sequence of support cohomology sequence;
\begin{equation}\label{support-sequence for X, X'-higher dimension}
   H_{Z - Z_{1}}^{i}(X_{1}, n) \rightarrow H_{\et}^{i}(X_{1}, n) \rightarrow H_{\et}^{i}(X, n) \rightarrow H_{Z - Z_{1}}^{i+1}(X_{1}, n).  
   \end{equation}
By Gabber purity isomorphism, we have
\begin{equation}
    H_{Z - Z_{1}}^{i}(X_{1}, n) \cong \oplus_{\alpha \in I_{1}} H_{Z_{1}^{\alpha}}^{i}(X_{1}, n) \cong  H_{\et}^{i-2c_{\alpha}}(Z^{\alpha}_{1}, n-c_{\alpha}), \hspace{2mm} \text{and}
\end{equation}
\begin{equation}
    H_{Z - Z_{1}}^{i+1}(X_{1}, n) \cong \oplus_{\alpha \in I_{1}} H_{Z_{1}^{\alpha}}^{i+1}(X_{1}, n) \cong  H_{\et}^{i+1-2c_{\alpha}}(Z^{\alpha}_{1}, n-c_{\alpha}).
\end{equation}
Since $Z^{\alpha}_{1}$ are quasi-projective varieties over $k_{N}$ of dimension less than $d$,  the induction hypothesis implies that  the group $H_{\et}^{i-2c_{\alpha}}(Z^{\alpha}_{1}, n-c_{\alpha})$ is finite if either $i-2c_{\alpha} \nin [2(n-c_{\alpha})- (d-c_{\alpha})-N, (n-c_{\alpha}) + (d-c_{\alpha}) + 1]$ or $n - c_{\alpha} \nin [0, (d-c_{\alpha})+N]$, which is equivalent to either $i \nin [2n-(d-c_{\alpha})-N , n+d+1]$  or $n \nin [c_{\alpha}, d + N]$. Similarly, the group $H_{\et}^{i+1-2c_{\alpha}}(Z^{\alpha}_{1}, n-c_{\alpha})$ is finite if either $i+1-2c_{\alpha} \nin [2(n-c_{\alpha})- (d-c_{\alpha})- N, (n-c_{\alpha}) + (d-c_{\alpha}) + 1]$ or $n - c_{\alpha} \nin [0, (d-c_{\alpha})+N]$, which is equivalent to either $i \nin [2n-(d-c_{\alpha})-N-1, n+d]$  or $n\nin [c_{\alpha}, d + N]$. This conditions are satisfied whenever  $i \nin [2n-N-1, n+d]$  or $n\nin [0, d + N]$ as $d-c_{\alpha} \leq d-1$ for each $c_{\alpha}$.  It follows from (\ref{support-sequence for X, X'-higher dimension}) that the following map
\begin{center}
   $ H_{\et}^{i}(X_{1}, n) \rightarrow H_{\et}^{i}(X, n)$
\end{center}
has finite kernel and cokernel whenever $i$ and $n$ satisfy $(1)$ or $(2)$. We note that $X_{1}$ is an open subscheme of regular scheme $X'$, with complement $X' - X_{1} = Z_{1}$, where dim($Z_{1}) < $ dim($Z$). Repeating the above steps by replacing $X$ by $X_{1}$, we can find $X_{2}$ and $Z_{2}$ such that $X_{1}$ is an open subscheme of $X_{2}$, dim($Z_{2}) < $ dim($Z_{1}$), and the natural map
\begin{center}
   $ H_{\et}^{i}(X_{2}, n) \rightarrow H_{\et}^{i}(X_{1}, n)$
\end{center}
has finite kernel and cokernel whenever $i$ and $n$ satisfy $(1)$ or $(2)$. Repeating this process finitely many times, we can find $X_{k}$ and $Z_{k}$ such that $X_{k-1}$ is an open subscheme of $X_{k}$, dim($Z_{k}) = 0$, and the natural map
\begin{center}
   $ H_{\et}^{i}(X_{k}, n) \rightarrow H_{\et}^{i}(X_{k-1}, n)$
\end{center}
has finite kernel and cokernel whenever $i$ and $n$ satisfy $(1)$ or $(2)$. Since dim($Z_{k}) = 0$, we have $Z_{k+1} = \emptyset $, and hence $X_{k+1} = X'$.  By composing all $(k+1)$-maps, we conclude that the map
   \begin{center}
   $ H_{\et}^{i}(X', n) \rightarrow H_{\et}^{i}(X, n)$
\end{center}
has finite kernel and cokernel whenever $i$ and $n$ satisfy $(1)$ or $(2)$. This completes the proof.
\end{proof}

\begin{lem}\label{existence of open subscheme}
    Let $N \geq 1$ be any integer, and let $X$ be a regular quasi-projective variety  of dimension $d$ over $k_{N}$. Assume that Theorem \ref{quasi-projective-N-local field-bounded} holds for all quasi-projective varieties over $k_{N}$ of dimension less than $d$.  Then there exists an open dense subscheme $U$ of $X$ such that  $H_{\et}^{i} (U, n) $  is finite if any one of the following holds;
    \begin{enumerate}
        \item  $i \nin [2n-N-d, n+d+1]$;
        \item $n \nin [0, d+N]$.
    \end{enumerate}

    \end{lem}

\begin{proof}

 Without loss of generality, one can assume that $X$ is integral. Let $X$ be a regular quasi-projective integral variety over $k$ and let 
    $X \inj \overline{X}$ be a compactification of $X$.  Let $\pi\colon W \to \overline{X}$ be an alteration morphism of de Jong and Gabber \cite[Theorem 4.3.1]{Temkin}(also see \cite[Lemma 2.3]{GR}). In particular, the morphism $\pi$ is generically finite of degree $e$, where $e =1$ if ${{\rm Char}}(k_{N})=0$ and $e=p^r$ if ${{\rm Char}}(k_{N})=p>0$. (Note that by a generically finite map of degree $1$, we mean a birational map.) 
    Let $V \inj \overline{X}$ be a dense open subset of $\overline{X}$ such that the induced map $\pi_{V'}\colon V' = \pi^{-1}(V) \to V$ is finite and flat of degree $e$. Since $X$ is integral, the intersection $U = V \cap X$ is a dense open subset of $X$. We let $\pi_{U'} \colon U'=\pi^{-1}(U)\to U$ denote the induced map. Since $\pi_{V'}$ is finite and flat of degree $e$, it follows that the morphism $\pi_{U'}$ is also finite and flat of degree $e$. Therefore, the map $(\pi_{U'})^*\colon H_{\et}^i(U, n) \to H_{\et}^i(U', n)$ is injective as $ (\pi_{U'})_* \circ \pi_{U'}^* =  e$ (see (\ref{eqn:Real-iso-2})). We claim that $U$ is the desired subscheme in the lemma. To show this, it is enough to show that $H_{\et}^i(U', n)$ is finite whenever $i$ and $n$ satisfy the given conditions.  Since $U'$ is a dense open subscheme of a regular projective scheme $W$, by Proposition \ref{finiteness can be shrinked for higher dimension}, we are reduce to show that the desired finiteness result for $W$. But this follows from Theorem \ref{projective-N-local field-bounded-restate}, which completes the proof.
    \end{proof}
   
\begin{cor}\label{cor:finiteness for X embedded inside projective X'}
    Let $N \geq 1$ be any integer. Suppose $X'$ is a regular quasi-projective variety of dimension $d \geq 1$ over $k_{N}$, 
    Assume that Theorem \ref{quasi-projective-N-local field-bounded} holds for all quasi-projective varieties over $k_{N}$ of dimension less than $d$. Then $ H_{\et}^{i} (X', n)$ is finite if any one of the following holds;
    \begin{enumerate}
        \item  $i \nin [2n-N-d, n+d+1]$;
        \item $n \nin [0, d+N]$.
    \end{enumerate}
       \end{cor}
\begin{proof}
   By Lemma \ref{existence of open subscheme}, there exists a dense open  subscheme $X$ of $X'$ such that $ H_{\et}^{i} (X, n)$ is finite whenever $i$ and $n$ satisfy the conditions $(1)$ or $(2)$. We now apply Lemma \ref{finiteness can be shrinked for higher dimension} for the pair $(X, X')$ to conclude that $H_{\et}^{i} (X', n)$ is finite whenever $i$ and $n$ satisfy the conditions $(1)$ or $(2)$. This completes the proof.
\end{proof}

\begin{thm}\label{thm:finiteness for quasiprojective under inductive assumption}
    Let $N \geq 1$ be any integer. Suppose $X'$ is a quasi-projective variety of dimension $d \geq 1$ over $k_{N}$, 
    Then $ H_{\et}^{i} (X', n)$ is finite if any one of the following holds;
    \begin{enumerate}
        \item  $i \nin [2n-N-d, n+d+1]$;
        \item $n \nin [0, d+N]$.
    \end{enumerate}
\end{thm}
\begin{proof}
    The case $d=1$ is covered by Theorem~\ref{quasi-projective-bounded-curve}. 
We therefore assume that $d \geq 2$ and proceed by induction on $d$. Let $X'$ be a quasi-projective variety of dimension $d$. 
By proceeding \emph{mutatis mutandis} as in the proof of Theorem~\ref{key theorem} 
and using the induction hypothesis, we reduce to the case where $X$ is a regular quasi-projective variety over $k_N$.
The result then follows by combining the induction hypothesis with Corollary~\ref{cor:finiteness for X embedded inside projective X'}. This completes the proof.
\end{proof}

\section{Étale $K$-groups over higher local fields} \label{sec:Etale-N-local}
In this section, we study the étale $K$-groups for varieties and prove certain structure of étale $K$-groups for quasi-projective varieties over higher local fields. We  use this structure to prove analogous results for algebraic $K$-groups in the later part of the paper.

\'{E}tale $K$-theory is closely related to algebraic $K$-theory via the
Quillen--Lichtenbaum conjectures. One of the advantages of \'{e}tale
$K$-theory over algebraic $K$-theory is the existence of an
Atiyah--Hirzebruch--type spectral sequence (\cite[Proposition 5.2]{Dwyer-Friedlander-EtaleK}) whose $E_{2}^{p,q}$-terms are
given by \'{e}tale cohomology and which converges to \'{e}tale $K$-groups.
In this section, we exploit this spectral sequence to compute \'{e}tale
$K$-groups of varieties over higher local fields, using the known results
on the \'{e}tale cohomology of such varieties established in the previous
sections.

Let $k$ be a field and let $l$ be a prime invertible in $k$ such that
$\mathrm{cd}_{l}(k) < \infty$. Let $X$ be a Noetherian scheme of finite type
over $k$. The \'{e}tale $K$-theory spectrum of $X$, denoted by
$\widehat{\mathbf{K}}^{\et}_{X}$, is defined in
\cite[Definition~14.3]{Dwyer-Friedlander-EtaleK} (also see \cite{Friedlander}). The \'{e}tale $K$-theory
groups $K^{\et}_{i}(X)$ are defined, for $i \ge 0$, as the $i^{th}$ homotopy
group of the spectrum $\widehat{\mathbf{K}}^{\et}_{X}$. For $v \ge 1$, the \'{e}tale $K$-theory groups with coefficients
$\mathbb{Z}/l^{v}$, denoted by $K^{\et}_{i}(X,\mathbb{Z}/l^{v})$, are defined
for $i \ge 0$ as the $i^{th}$ homotopy group of the spectrum
$\widehat{\mathbf{K}}^{\et}_{X} \wedge \mathcal{M}(l^{v})$, where
$\mathcal{M}(l^{v})$ denotes the mod $l^{v}$ Moore spectrum
(see \cite[Definition~4.3]{Dwyer-Friedlander-EtaleK}). The main result of this section is the following.

\begin{thm}\label{thm:structure of etale-K-group over N-local}
    Let $X$ be connected quasi-projective variety of dimension $d$ over $k_{N}$. Then for any integer $m \geq d+N+1$, we have
    \begin{center}
        $K^{\et}_{m}(X) \cong F \oplus D$,
    \end{center}
    where $F$ is a finite and $D$ is uniquely $p'$-divisible group.
\end{thm}

We first recall a well-known spectral sequence, which we use  to achieve the main results of this section.

\begin{thm}\cite[Proposition 5.2]{Dwyer-Friedlander-EtaleK}\label{spectral-sequence-etale coh to etale K}
	Let $X$ be a connected Noetherian scheme over $\Z/l$ of finite $\mathbb{Z}/l$-cohomological dimension. Then there is a natural strongly convergent, fourth quadrant spectral sequence
	\[
E^{p,-q}_{2} = H_{\et}^{p} (X, \mathbb{Z}/l^{v}(q/2) ) \Rightarrow K^{\et}_{q-p}(X, \mathbb{Z}/l^{v}),
\]
where 	$\mathbb{Z}/l^{v}(\frac{q}{2})$ is $0$ unless $q$ is a non-negative even integer.

\end{thm}

\begin{prop}\label{prop:K^et-bounded-with-finite coefficient}
     Let $X$ be connected quasi-projective variety of dimension $d$ over $k_{N}$. Then for any given integer $m \geq d+N+1$, there exists an integer $M$ (depending upon $m$) such that we have
    \begin{center}
        $  \lvert  K^{\et}_{m}(X, \mathbb{Z}/l^{v}) \rvert \leq M$
    \end{center}
    for all primes $l \neq p$ and all $v \geq 1$.
\end{prop}

\begin{proof}
    By applying Proposition \ref{spectral-sequence-etale coh to etale K}, we have
     \[
E^{p,-q}_{2} = H_{\et}^{p} (X, \mathbb{Z}/l^{v}(q/2) ) \Rightarrow K^{\et}_{q-p}(X, \mathbb{Z}/l^{v}).
\] 
It is enough to show that there exists an integer $M$ such that $\lvert E^{p,-q}_{2} \rvert \leq M $ whenever $q-p \geq d+N+1$, for all primes $l \neq p$ and every $v \geq 1$. Indeed, the condition $q-p \geq d+N+1$ is equivalent to $p \leq q-d -N-1$. This implies that $p \nin [q-d-N, \frac{q}{2}+d+2]$, and hence by Theorem \ref{thm:finiteness for quasiprojective over N-local fields}, we obtain
\begin{center}
   $ \lvert H_{\et}^{p} (X, \mathbb{Z}/l^{v}(q/2) ) \rvert \leq M$
\end{center}
for every prime $l \neq p$ and every $v \geq 1$. Fix any $q$ such that $q-p \geq d+N+1$. Then we have $  \lvert E^{p,-q}_{2} \rvert \leq M $, and consequently, 
\begin{center}
    $\lvert E^{p,-q}_{r} \rvert \leq M $ for all $r \geq 2$.
\end{center}
In particular, we have $\lvert E^{p,-q}_{\infty} \rvert \leq M $. Therefore, for any integer $m$ such that $m\geq d+N+1$ with $q -p = m$, we have $\lvert E^{p,-q}_{\infty} \rvert \leq M $. This implies that 
\begin{center}
    $ \lvert K^{\et}_{m}(X, \mathbb{Z}/l^{v}) \rvert \leq M $ 
\end{center}
for every prime $l \neq p$ and every $v \geq 1$. This completes the proof.
\end{proof}

We denote by $\mathbb{L}$ the set of natural numbers that are not divisible by $p$, where $p$ is the characteristic of $k_{0}$.

\begin{prop}\label{prop:K^et-bounded-with-finite coefficient for n}
     Let $X$ be connected quasi-projective variety of dimension $d$ over $k_{N}$. Then for any given integer $m \geq d+N+1$, there exists an integer $M$ (depending upon $m$) such that
    \begin{center}
        $  \lvert  K^{\et}_{m}(X, \mathbb{Z}/n) \rvert \leq M$
    \end{center}
    for all $n \in \mathbb{L}$.
\end{prop}

\begin{proof}
 Let $n \in \mathbb{L}$ be any integer, and we write  $n = l^{v_{1}}_{1}l^{v_{2}}_{2}...l^{v_{s}}_{s}$  for some integer $s \geq 1$, where $ l_{1}, l_{2}, ..., l_{s}$ are distinct primes such that $ p \neq l_{i}$ for any $i \in \{ 1, 2, ..., s\}$. We let  $E^{p,-q}_{2}(n)  := H_{\et}^{p} (X, \mathbb{Z}/n(q/2) )$. Then we have
 \begin{equation}
     E^{p,-q}_{2}(n)  = E^{p,-q}_{2}(l^{v_{1}}_{1}) \oplus \cdots \oplus E^{p,-q}_{2}(l^{v_{s}}_{s}), \hspace{2mm}  \text{and} 
 \end{equation}
 \begin{equation}
     K^{\et}_{m}(X, \mathbb{Z}/n)  = K^{\et}_{m}(X, \mathbb{Z}/l^{v_{1}}_{1}) \oplus \cdots \oplus K^{\et}_{m}(X, \mathbb{Z}/l^{v_{s}}_{s}).
 \end{equation}
for any integer $m$. 
By Theorem \ref{thm:finiteness for quasiprojective over N-local fields} (also see Proposition \ref{prop:K^et-bounded-with-finite coefficient}), there exists an integer $M$ such that
\begin{center}
    $ \lvert  E^{p,-q}_{2}(n) \rvert \leq M$, whenever $q-p \geq d+N+1$.
\end{center}
This implies that 
\begin{center}
    $ \lvert  E^{p,-q}_{2}(n) \rvert  = \lvert E^{p,-q}_{2}(l^{v_{1}}_{1}) \rvert \cdots \lvert E^{p,-q}_{2}(l^{v_{s}}_{s}) \rvert \leq M$, whenever $q-p \geq d+N+1$.
\end{center}
 Now by mimicking the similar arguments as in the proof of Proposition \ref{prop:K^et-bounded-with-finite coefficient}, we have
 \begin{center}
    $ \lvert K^{\et}_{m}(X, \mathbb{Z}/n) \rvert  = \lvert K^{\et}_{m}(X, \mathbb{Z}/l^{v_{1}}_{1}) \rvert \cdots  \lvert K^{\et}_{m}(X, \mathbb{Z}/l^{v_{s}}_{s})  \rvert \leq M$,
\end{center}
  whenever $q-p \geq d+N+1$ with $q-p=m$. This completes the proof.
\end{proof}

\begin{cor}\label{cor: etale-K theory-mod and torsion boundedness} Let $X$ be a connected quasi-projective variety of dimension $d$ over $k_{N}$. Then for any given integer $m \geq d+N+1$, there exists an integer $M$ (depending upon $m$) such that

\begin{center}

$ \lvert  K^{\et}_{m}(X)/n \rvert \leq M$ and  \hspace{1mm}$ \lvert  _nK^{\et}_{m-1}(X) \rvert \leq M$.
 
   \end{center}
for all $n \in \mathbb{L}.$

\end{cor}

\begin{proof}
        
 Let $n \in \mathbb{L}$ be any integer.  The corollary follows from Proposition \ref{prop:K^et-bounded-with-finite coefficient for n} and the short exact sequence 
\begin{equation}\label{equation-K-et}
0 \rightarrow K^{\et}_{m}(X)/n \rightarrow K^{\et}_{m}(X, \mathbb{Z}/n) \rightarrow  _n K^{\et}_{m-1}(X) \rightarrow 0.
\end{equation}
\end{proof}

\begin{cor}\label{finitess of p'-torsion in etale K theory}
    Let $X$ be as in Corollary \ref{cor: etale-K theory-mod and torsion boundedness}. Then for $m \geq d +N$,  the $p'$-torsion subgroup of $K^{\et}_{m}(X) $ is finite i.e., $K^{\et}_{m}(X) \{ p' \} $ is finite.
\end{cor}

\begin{proof}
  For  $m \geq N+d$, Corollary \ref{cor: etale-K theory-mod and torsion boundedness} yields that $ \lvert  _nK^{\et}_{m}(X)  \rvert \leq M$  for all $n \in \mathbb{L}$. Taking the direct limit over all integers in $\mathbb{L}$, we get the desired result.
\end{proof}

Recall that we say an abelian group $A$ is $\mathbb{L}$-divisible if $A/n = 0$ for all $ n \in \mathbb{L}$. This is equivalent to saying that $A$ is $p'$-divisible group. We now recall the following result from group theory, which is well known to experts.

\begin{lem}\label{lem: group-theory-bounded}
     Let $A$ be an abelian group. Suppose that
      there exists an integer $M \gg 0$ such that $|{A}/{m}| \le M$ for all
      $m \in \mathbb{L}$. Then $A \cong F \bigoplus D$, where $F$ is finite group of
      order lying in $\mathbb{L}$ and $D$ is $\mathbb{L}$-divisible.
\end{lem}

\begin{proof}
The proof is an immediate consequence of \cite[Corollary 11.4]{GKR}, upon taking $\mathbb{M}$ to be the set of all primes different from $p$ in their argument.
\end{proof}

We now begin the proof of Theorem \ref{thm:structure of etale-K-group over N-local}.

\begin{proof}
Let $m \geq N+d+1$ be any given integer. We first claim that  \begin{center}
        $K^{\et}_{m}(X) \cong F \oplus D$,
    \end{center}
    where $F$ is a finite and $D$ is $p'$-divisible group. By Lemma \ref{lem: group-theory-bounded}, this claim is equivalent to showing that there exists an integer $M$ such that $ \lvert  K^{\et}_{m}(X)/n \rvert \leq M$ for all $n \in \mathbb{L}$. But this follows from Corollary \ref{cor: etale-K theory-mod and torsion boundedness}. This completes the proof of the claim. It now remains to show that $D$ is uniquely $p'$-divisible, which is equivalent to show $D  \lbrace p' \rbrace = 0$. First note that the group is $D  \lbrace p' \rbrace$ is $p'$-divisible as $D$ is $p'$-divisible. By applying  Corollary \ref{finitess of p'-torsion in etale K theory}, we get that $K^{\et}_{m}(X) \{ p' \}$ is finite, which in particular implies that $D  \lbrace p' \rbrace $ is finite. The group $D \{ p'\}$ is therefore finite $p'$-torsion and $p'$-divisible group, and hence $D \{p'\} = 0$. This completes the proof of  Theorem \ref{thm:structure of etale-K-group over N-local}. \end{proof}


\section{Algebraic $K$-groups over higher local fields} \label{sec:Alg-K-N-local}
In this section, we prove one of the main results of this paper, which describe the structure of algebraic $K$-groups of quasi-projective varieties over higher local fields. An analogous structure of $K$-groups for smooth varieties over archimedian local fields such as complex numbers was proved by Weibel-Pedrini \cite{Weibel-Pedrini-surface}, \cite{Weibel-Pedrini}. We first describe the structure of algebraic $K$-groups of quasi-projective varieties over $N$-local fields. We then specialize to smooth varieties over finite fields, where we prove a stronger result that also allows us to relate the Bass conjecture and the Parshin conjecture in higher degrees.

  Let $X$ be a quasi-projective variety over a field $k$. The algebraic $K$-groups of $X$ were defined by D. Quillen. We denote by $K_i(X)$ the $i^{th}$ algebraic $K$-group of $X$ in the sense of Quillen \cite{Quillen}; by definition, $K_i(X)$ is the $i^{th}$ homotopy group of the $K$-theory spectrum $\mathbf{K}(X)$. Similarly, one can define the $K$-groups with finite coefficient say $\mathbb{Z}/n$ for any integer $n \geq 2$, denoted as $K_{i}(X, \mathbb{Z}/n)$ (see \cite[IV, Definition 2.4]{Weibel-K}).
  

  We  denote by $k_{N}$ a $N$-local field and we let $p$ denote the characteristic of of the associated finite field $k_{0}$. We describe the structure of algebraic $K$-groups of quasi-projective varieties in the following theorem. 

  \begin{thm}\label{thm:structure of algebraic K-group over N-local}
    Let $X$ be a connected quasi-projective variety of dimension $d \geq 0$ over $k_{N}$. Then for any integer $m \geq d+N+1$, we have
    \begin{center}
        $K_{m}(X) \cong F \oplus D$,
    \end{center}
    where $F$ is a finite and $D$ is a uniquely $p'$-divisible group. In particular, $K_{m}(X)\{p'\}$ is finite.  Moreover, if ${\rm char}(k_N) = {\rm char}(k_0)$, then $D$ is a uniquely divisible group and $K_{m}(X)_{\tor}$ is finite. 
\end{thm}

One of the key results we need to prove the above theorem is the theorem of Rost, Voevodsky \cite{V-Annals}, Weibel which was initially a conjecture known as the Quillen-Lichtenbaum Conjecture. Later the result was shown for arbitrary quasi-projective varieties by Clausen-Mathew (see \cite[Theorem 1.2]{Mathew-Clausen}). This result allow us to use results of étale $K$-groups to algebraic $K$-groups in large enough degree.

\begin{thm} (\cite{V-Annals}, \cite{Mathew-Clausen})\label{Quillen Lichtenbaum} 
	Let $X$ be a quasi-projective variety over $k$ and let  $n$ be a positive integer invertible on $X$. Then for any integer $m \geq d + cd(k)$, the map
    
    \begin{center}
        $K_{m}(X, \mathbb{Z}/n) \rightarrow K^{\et}_{m}(X, \mathbb{Z}/n)$ 
    \end{center}
 is an isomorphism.    
\end{thm}

\begin{prop}\label{prop:algebraic K-group-bounded-with-finite coefficient}
     Let $X$ be a connected quasi-projective variety of dimension $d$ over $k_{N}$. Then for any given integer $m \geq d+N+1$, there exists an integer $M$ (depending upon $m$) such that we have
     \vspace{1mm}
\begin{enumerate}
        \item 
    
        $  \lvert  K_{m}(X, \mathbb{Z}/n) \rvert \leq M$  $\forall \hspace{2mm} n \in \mathbb{L}$. 

        \item $ \lvert  K_{m}(X)/n \rvert \leq M$ $\forall  \hspace{2mm}$ $n \in \mathbb{L}$. 
        
        \item  $\lvert  _nK_{m-1}(X) \rvert \leq M$ $\forall  \hspace{2mm}$  $ n \in \mathbb{L}$.

   \end{enumerate}

\end{prop}

\begin{proof}
     The parts $(2)$ and $(3)$ follow from $(1)$ by using the following short exact sequence (see \cite[Page. 306]{Weibel-K}).
\begin{equation}\label{equation-K-group-UCT}
0 \rightarrow K_{m}(X)/n \rightarrow K_{m}(X, \mathbb{Z}/n) \rightarrow  {}_{n} K_{m-1}(X) \rightarrow 0.
\end{equation}
It now remains to show $(1)$. By Theorem \ref{Quillen Lichtenbaum}, we have an isomorphism  
\begin{center}
        $K_{m}(X, \mathbb{Z}/n) \cong K^{\et}_{m}(X, \mathbb{Z}/n)$ 
    \end{center} for $m \geq d + N+1$, and for all $n \in \mathbb{L}$. Hence it suffices to show that the similar claim for étale $K$-groups in the given range. This follows from Proposition \ref{prop:K^et-bounded-with-finite coefficient for n}, which completes the proof.
\end{proof}

    \begin{cor}\label{finitess of p'-torsion in algebraic K theory}
    Let $X$ be as in Proposition \ref{prop:algebraic K-group-bounded-with-finite coefficient}. Then for $m \geq d +N$,  the $p'$-torsion subgroup of $K_{m}(X) $ is finite i.e., $K_{m}(X) \{ p' \} $ is finite.
\end{cor}

\begin{proof}
    For $m \geq N+d$, Proposition \ref{prop:algebraic K-group-bounded-with-finite coefficient}(3) yields that $ \lvert  _nK_{m}(X)  \rvert \leq M$  for all $n \in \mathbb{L}$. Taking the direct limit over all integers $n$ in $\mathbb{L}$, we get the desired result.
\end{proof}

We prove with the following lemma before turning to the proof of \thmref{thm:structure of algebraic K-group over N-local}. The following lemma is motivated from a similar result of Geisser and Levine
 \cite[Last claim of Theorem~8.4]{GL} over perfect fields.

 \begin{lem} \label{lem:Gei-Lev-p-div}
     Let $X$ be a smooth variety of dimension $d \geq 0$ over $k_{N}$ of characteristics $p > 0$. Then $K_{m}(X, \mathbb{Z}/p) = 0$ if $m \geq d + N +1$. 
 \end{lem}
\begin{proof}
   We have the following convergent spectral sequence (see \cite[Page~465]{GL}).  
    \begin{equation*}
         E^{i,j}_2  = H^i(X_{\rm Zar},(\mathcal{K}/p)_{-j} ) \implies K_{-i-j}(X,Z/p),
 \end{equation*}
    where $(\mathcal{K}/p)_{-j}$ is the Zariski sheaf associated to the presheaf $(U \mapsto K_{-j}(U, \Z/p))$. It therefore suffices to show that  the groups $ H^i(X_{\rm Zar},(\mathcal{K}/p)_{-j} )$ are zero when $-i-j \geq d + N +1$. Since these groups vanish for $i < 0$, it suffices to show that the sheaf $(\mathcal{K}/p)_{-j}$ vanish when
    $-j \geq d+N+1$. Since the Gersten resolution of the sheaf $(\mathcal{K}/p)_{-j}$ is universally exact (see \cite{Grayson}), it suffices to show that 
    $K_{-j}(k(\eta))/p = 0$ for $-j \geq d+N+1$, where $\eta$ is a generic point of $X$. To prove this, we first note that $K_{-j}(k(\eta))/p \cong K^M_{-j}(k(\eta))/p = \Omega^{-j}_{k(\eta), {\rm log}} \inj \Omega^{-j}_{k(\eta)}$. We finally observe that  $\Omega^{-j}_{k(\eta)} =0 $ for $-j \geq d+N+1$ because $[k(\eta): k(\eta)^p] = d + [k_N: k_N^p] = d+N$. This completes the proof of the lemma. 
\end{proof}

We now prove Theorem \ref{thm:structure of algebraic K-group over N-local}.

\begin{proof}
Let $m \geq d+N+1$ be any given integer. We first claim that  \begin{center}
        $K_{m}(X) \cong F \oplus D$,
    \end{center}
    where $F$ is a finite group and $D$ is a $p'$-divisible group. By Lemma \ref{lem: group-theory-bounded}, this claim is equivalent to show that there exists an integer $M$ such that $ \lvert  K_{m}(X)/n \rvert \leq M$ for all $n \in \mathbb{L}$. This follows from Proposition \ref{prop:algebraic K-group-bounded-with-finite coefficient}, which completes the proof of the claim. It remains to show that $D$ is uniquely $p'$-divisible, which is equivalent to show $D  \lbrace p' \rbrace = 0$. First note that the group  $D  \lbrace p' \rbrace$ is $p'$-divisible. By applying Corollary \ref{finitess of p'-torsion in algebraic K theory}, we get that $K_{m}(X) \{ p' \}$ is finite, which in particular implies that $D  \lbrace p' \rbrace $ is finite. The group $D \{ p'\}$ is therefore a finite $p'$-torsion and $p'$-divisible group, and hence $D \{p'\} = 0$. This completes the proof of the first part of \thmref{thm:structure of algebraic K-group over N-local}. We now assume that $k_N$ is an equi-characteristic $N$-local fields, i.e., 
    ${\rm char}(k_N) = {\rm char}(k_0)$. 
     By \lemref{lem:Gei-Lev-p-div}, we have  
     $K_{m}(X, \mathbb{Z}/p) = 0$ for $m \geq  d+N+1$. We use the short exact sequence
\begin{equation}
	0 \rightarrow K_{m}(X)/p \rightarrow K_{m}(X, \mathbb{Z}/p) \rightarrow  _pK_{m-1}(X) \rightarrow 0,	
	\end{equation}
	to conclude that $K_{m}(X)$ is $p$-divisible whenever $m \geq d+N+1$ and has no $p$-torsion when $m \geq d+N$. Since the order of $F$ is not divisible by $p$  (see Lemma \ref{lem: group-theory-bounded}),  $F$ is also uniquely $p$-divisible. This implies that $D$ is uniquely $p$-divisible whenever $m \geq d +N+ 1$. This completes the proof of the theorem. 
    \end{proof}

\section{$K$-groups for smooth varieties over finite fields and local fields} 
\label{sec:Alg-K-groups-Finite-field}

In this section, we strengthen the structure theorems for étale $K$-groups and algebraic $K$-groups (see Theorems~\ref{thm:structure of etale-K-group over N-local} and~\ref{thm:structure K-group for smooth over finite fields}) in the case where $X$ is a smooth projective variety and the base field is either a finite field or a local field.

\subsection{Étale $K$-groups}
Throughout this subsection, we let $k$ to be either a finite field of characteristic $p>0$ or a local field of residue characteristic $p>0$ and let $X$ be a smooth projective variety of dimension $d \geq 0$ over $k$. We denote by $\mathbb{L}$ the set of natural numbers that are not divisible by $p$.

\begin{prop}\label{etale cohomology bounded for finite and local}
     Let $X$ be a smooth projective variety of dimension $d\geq 0$ over $k$. 
     \begin{enumerate}
         \item If $k$ is finite, then there exists an integer $M$ such that
         \begin{center}
              $ \lvert H_{\et}^{i} (X, \mathbb{Z}/m(n) ) \rvert \leq M, \hspace{2mm}  \forall \hspace{2mm}  i \nin 2n, 2n+1, 2n+2$ and $\forall$ $m \in \mathbb{L}$.
         \end{center}
        
         \item If $k$ is local, and $X$ admits a good reduction, then there exists an integer $M'$
         \begin{center}
              $ \lvert H_{\et}^{i} (X, \mathbb{Z}/m(n) ) \rvert \leq M', \hspace{2mm}  \forall \hspace{2mm}  i \nin 2n-1, 2n, 2n+1, 2n+2$ $\forall$ $m \in \mathbb{L}$.
         \end{center}

     \end{enumerate}
\end{prop}
     
\begin{proof}
    For any integer $m \in \mathbb{L}$, we have the exact sequence of sheaves
  \[
    {\Z}/m(n) \to ({\Q}/{\Z})'(n) \xrightarrow{m} ({\Q}/{\Z})'(n),
  \]
  where $({\Q}/{\Z})' = {\underset{\ell \neq p}\bigoplus} {\Q_\ell}/{\Z_\ell}$ and we have the following long exact sequence.
\begin{equation}
      \cdots \rightarrow H_{\et}^{i-1}(X, n) \rightarrow H_{\et}^{i}(X, \mathbb{Z}/m(n)) \rightarrow H_{\et}^{i}(X, n)  \rightarrow \cdots
      \end{equation}
Taking the limit over integers $m$ such that $(m,p) = 1$, part $(1)$ follows from Theorem \ref{finiteness-smooth-finite field}, and part $(2)$ follows from Corollary \ref{good-reduction-finiteness}.
\end{proof}

\begin{prop}\label{prop:K^et-bounded-with-finite coefficient for n for smooth}
     Let $X$ be smooth projective connected variety of dimension $d$ over $k$.
    \begin{enumerate}
        \item If $k$ is finite, then for any given integer $m \geq 1$, there exists an integer $M$ such that  $  \lvert  K^{\et}_{m}(X, \mathbb{Z}/n) \rvert \leq M$  for all $n \in \mathbb{L}$. 
        \vspace{2mm}

        \item If $k$ is local, and $X$ admits a good reduction, then for any given integer $m \geq 2$, there exists an integer $M'$  such that $  \lvert  K^{\et}_{m}(X, \mathbb{Z}/n) \rvert \leq M'$  for all $n \in \mathbb{L}$.

    \end{enumerate}
   
  \end{prop}  

\begin{proof}
   \begin{enumerate}
\item[(1)] The proof follows the same steps as in Proposition~\ref{prop:K^et-bounded-with-finite coefficient for n} (also see Proposition~\ref{prop:K^et-bounded-with-finite coefficient}), except that we use Proposition~\ref{etale cohomology bounded for finite and local}$(1)$, which yields
\[
\lvert H_{\et}^{i}(X, \mathbb{Z}/n(j/2)) \rvert \le M \quad \text{whenever } j - i \ge 1,
\]
instead of Theorem~\ref{thm:Main:structure of algebraic-K-group over N-local}. The remainder of the proof proceeds along the same lines.

\item[(2)] The proof of (2) also follows via similar steps once we use Proposition~\ref{etale cohomology bounded for finite and local}$(2)$, which yields
\[
\lvert H_{\et}^{i}(X, \mathbb{Z}/n(j/2)) \rvert \le M \quad \text{whenever } j - i \geq 2,
\]
instead of Theorem~\ref{thm:Main:structure of algebraic-K-group over N-local} in Proposition~\ref{prop:K^et-bounded-with-finite coefficient for n}.

\end{enumerate}
This completes the proof.
\end{proof}

\begin{cor}\label{cor: etale-K theory-mod and torsion boundedness for smooth} Let $X$ be a connected smooth projective variety of dimension $d$ over $k$. For any given integer $m \geq 1$ (if $k$ is finite), and $m \geq 2$ (if $k$ is local and $X$ has good reduction), there exist an integer $M$ (depends upon $k$) such that 
\begin{center}
     $ \lvert  K^{\et}_{m}(X)/n \rvert \leq M$, and  \hspace{1mm}$ \lvert  _nK^{\et}_{m-1}(X) \rvert \leq M$ for all $n \in \mathbb{L}.$
\end{center} 
\end{cor}
\begin{proof}
    Let $n \in \mathbb{L}$ be any integer.  Th corollary follows from Proposition \ref{prop:K^et-bounded-with-finite coefficient for n for smooth} and by the short exact sequence
\begin{equation}\label{equation-K-et-1}
0 \rightarrow K^{\et}_{m}(X)/n \rightarrow K^{\et}_{m}(X, \mathbb{Z}/n) \rightarrow  _n K^{\et}_{m-1}(X) \rightarrow 0.
\end{equation}
This completes the proof.
\end{proof}

\begin{cor}\label{finitess of p'-torsion in etale K theory for smooth}
   Let $X$ be a connected smooth projective variety of dimension $d$ over $k$. For any given integer $m \geq 0$ (if $k$ is finite), and $m \geq 1$ (if $k$ is local and $X$ has good reduction),  the $p'$-torsion subgroup of $K^{\et}_{m}(X) $ is finite i.e., $K^{\et}_{m}(X) \{ p' \} $ is finite.
\end{cor}

We now restate and prove \thmref{thm:main:structure of etale-K-group over N-local for smooth}. 

    \begin{thm}\label{thm:structure of etale-K-group over N-local for smooth}
    Let $X$ be a smooth projective connected variety of dimension $d$ over $k$.
\begin{enumerate}
    \item If $k$ is finite, then for any integer $m \geq 1$, we have
    \begin{center}
        $K^{\et}_{m}(X) \cong F \oplus D$,
    \end{center}
    where $F$ is a finite group and $D$ is a uniquely $p'$-divisible group.

\vspace{2mm}
     \item If $k$ is local, and $X$ admits a good reduction, then for any integer $m \geq 2$, we have
    \begin{center}
        $K^{\et}_{m}(X) \cong F' \oplus D'$,
    \end{center}
    where $F'$ is finite and $D'$ is a uniquely $p'$-divisible group.
\end{enumerate}

\end{thm}

\begin{proof}
    The proof proceeds along the same lines as the proof of Theorem~\ref{thm:structure of etale-K-group over N-local}. The only modification is that, in place of Corollary~\ref{cor: etale-K theory-mod and torsion boundedness} (resp. Corollary \ref{finitess of p'-torsion in etale K theory}), we use Corollary~\ref{cor: etale-K theory-mod and torsion boundedness for smooth} (resp. Corollary \ref{finitess of p'-torsion in etale K theory for smooth})  in the argument of Theorem~\ref{thm:structure of etale-K-group over N-local}. This completes the proof.
\end{proof}

\subsection{Algebraic $K$-groups} Throughout this subsection, we let $k$ to be a finite field of characteristic $p>0$ and let $X$ be a smooth projective variety of dimension $d \geq 0$ over $k$. Then the structure of $K$-groups of $X$ is conjecturally predicted by the following two deep conjectures.
\begin{conj}(Bass Conjecture)
	 The groups $K_{n}(X)$ are finitely generated for all integers $n \geq 0$.
\end{conj}

\begin{conj}(Parshin Conjecture)
	 $K_{n}(X)$ are torsion for all integers $n \geq 1$.	
\end{conj}

In particular, combining both conjectures, it predicts that $K_{n}(X)$ are finite for all integers $n \geq 1$.  In general, both conjectures remain open, except in the case where $\dim (X) \le 1$. The case $d=0$ was shown by Quillen \cite{Quillen-Annals}, while the case $d=1$ was proved by Quillen–Harder \cite[Chapter~VI, Theorem~6.1]{Weibel-K}. Here we prove the following.

      \begin{thm}\label{thm:structure K-group for smooth over finite fields}
       	Let $X$ be a smooth projective variety of dimension $d$ over a finite field $k$. Then for any $m \geq d + 1$, we have 
        \[
       K_{m}(X) \cong F \oplus D,
       \]
       where $F$ is a finite group and $D$ is a uniquely divisible group. Moreover, if $m = d$, then $K_{m}(X)_{\tor}$ is finite. In particular, $K_{m}(X)_{\tor}$ is finite if $ m \geq d$.
       \end{thm}

\begin{proof}
 Let $m \geq d+1$ be a given integer.   By Theorem \ref{thm:structure of algebraic K-group over N-local} (for $N = 0$), we know that \begin{equation}\label{Expression of Km(X)}
      K_{m}(X) \cong F \oplus D, 
 \end{equation}
       where $F$ is a finite group and $D$ is a $p'$-divisible group. Therefore it suffices to show that $D$ is uniquely $p$-divisible group. Note that $K_{m}(X, \mathbb{Z}/p) = 0$ for $m \geq  d+1$ (see \cite[Theorem 8.4]{GL}). We now use the short exact sequence
\begin{equation}
	0 \rightarrow K_{m}(X)/p \rightarrow K_{m}(X, \mathbb{Z}/p) \rightarrow  _pK_{m-1}(X) \rightarrow 0,	
	\end{equation}
	to conclude that $K_{m}(X)$ is $p$-divisible whenever $m \geq d+1$ and has no $p$-torsion when $m \geq d$. Since the order of $F$ is not divisible by $p$  (see Lemma \ref{lem: group-theory-bounded}),  $F$ is also uniquely $p$-divisible. This implies that $D$ is uniquely $p$-divisible whenever $m \geq d + 1$. This completes the proof of the first part of the claim. Also, we observed that $K_{m}(X)$ has no $p$-torsion if $m \geq d$.  By Corollary \ref{finitess of p'-torsion in algebraic K theory}, we have $K_{m}(X) \lbrace p' \rbrace $ is finite. Hence, we conclude that $K_{m}(X)_{\tor}$ is finite.  This completes the proof of  Theorem \ref{thm:structure K-group for smooth over finite fields} . 
\end{proof}

\begin{rem}
    Bass and Parshin conjectures predicts that the group $D$ in  Theorem \ref{thm:structure K-group for smooth over finite fields} vanishes.
\end{rem}
 One of the nice consequences of the above theorem is the following corollary, which shows the equivalence of Bass and Parshin’s conjectures in higher degrees. 
      
     \begin{cor}\label{equivalence of conjectures}
      	Let $X$ be a smooth projective variety of dimension $d$ over a finite field $k$. For any integer $m \geq  d + 1$, the following are equivalent.

\begin{enumerate}
    \item $K_{m}(X)$ is finite group of order not divisible by $p$.
    \item  $K_{m}(X)$ is finitely generated.
    \item $K_{m}(X)$ is torsion.
\end{enumerate}
     
  Moreover, if $K_{d}(X)$ is torsion, then $K_{d}(X)$ is finite.
      \end{cor}

\begin{proof}

By Theorem \ref{thm:structure K-group for smooth over finite fields}, we have  
\[
K_{m}(X) \cong F \oplus D,
\]
where $F$ is a finite group of order not divisible by $p$ and $D$ is a uniquely divisible group if $ m \geq d+ 1$. Also, we have $K_{d}(X)_{\tor}$ is finite. Now it is easy to conclude the proof as if $D$ is either finitely generated or torsion, then 
$D=0$. This completes the proof of Corollary \ref{equivalence of conjectures}.
  \end{proof}

\section{Finiteness of torsion in Bloch's higher Chow groups} \label{sec:Mot-HCG}

In this section, we study torsion in Bloch's higher Chow groups and prove Theorems~\ref{finiteness for chow groups over finite fields} and~\ref{thm:Main-3}. To this end, we begin by recalling motivic cohomology and identifying Bloch's higher Chow groups with motivic cohomology in Section \ref{subsection:motivic cohomology groups}. In the subsection Subsection \ref{subsection:finiteness of torsion in higher chow groups}, we use this identification together with finiteness results for \'etale cohomology from Section~\ref{section: quasi-proj over N-local} to prove Theorems~\ref{finiteness for chow groups over finite fields} and~\ref{thm:Main-3}. We also prove a result for finiteness of $p'$-torsion subgroup of higher Chow groups over $N$-local fields in Section \ref{torsion over N-local fields}.

\subsection{Recollection of Motivic cohomology groups}\label{subsection:motivic cohomology groups} 
Throughout the subsection, we fix a field $k$  of characteristic exponent $p \ge 1$. We let $\Lambda$ be a commutative ring which we assume to be $\Z$ if $p = 1$ or
any $\Z[\tfrac{1}{p}]$-algebra if $p \ge 2$. 

It was shown by Cisinski-D{\'e}glise \cite{CD-Springer} that
given a Noetherian $k$-scheme $X$ of finite Krull-dimension, there is a monoidal triangulated category of
mixed motives $\dm(X, \Lambda)$ which coincides with the original construction of
Voevodsky \cite{Voe00} when $X$ is the spectrum of a perfect field, and of Suslin
\cite{Suslin-AKT} when $X$ is the spectrum of an arbitrary field.
One defines $\un{\dm}_\cdh(X, \Lambda)$ 
in the same
manner by replacing the Nisnevich topology on $\Sm_X$ by the cdh topology on
$\Sch_X$. We let $\dm_\cdh(X)$ be the full localizing triangulated subcategory of
$\un{\dm}_\cdh(X, \Lambda)$ generated by motives of the form $M_X(Y)(n)$ for
$Y \in \Sm_X$ and $n \in \Z$. 
By the main result of \cite{CD-Doc}, the assignment $X \mapsto \dm_\cdh(X, \Lambda)$
satisfies Grothendieck's six functor formalism (see \cite[A.5]{CD-Springer})
in full generality in the category of Noetherian $k$-schemes of finite Krull
dimension. 
The constant  Nisnevich (resp. cdh) sheaf with transfer associated to the ring $\Lambda$ is the identity
object for the monoidal structure of $\dm(X, \Lambda)$ (resp. $ \dm_\cdh(X, \Lambda)$). We shall denote both of these by
$\Lambda_X$. 

Let $f \colon X \to \Spec(k)$ be the structure map,
recall the following definitions of motives in
$\dm_\cdh(k, \Lambda) \simeq \dm(k, \Lambda)$ (e.g., see \cite[\S~8.5]{CD-Doc}).
\begin{equation}\label{eqn:Motive-def}
  M_k(X) = f_\sharp \Lambda_X \cong f_! f^! \Lambda_k, \
\end{equation}
where all functors are considered on the premotivic category
$X \mapsto \dm_\cdh(X, \Lambda)$. Recall that $\Z_k(1) =
{\Z_{\rm tr}(\P^1_k)}/{\Z_{\rm tr}(\Spec(k))}$ as a presheaf with transfer
via the identification $\Spec(k) \cong k(\infty)$. In particular,
$\Lambda_k(1) \cong M_k({\P^1_k}/{\Spec(k)}) \in  \dm(k, \Lambda)$. 
We shall drop highlighting the base field $k$ in the notations of motives once it
is fixed in a context.

\begin{defn}\label{defn:MC-defn} For integers $q$ and $i$, we define the motivic cohomology and motivic cohomology with compact support (respectively) as follows.
\begin{equation}\label{eqn:motivic-Coh-defn}
  H_{M}^i(X, \Lambda(j)) = \Hom_{\dm_\cdh(k, \Lambda)}(M_k(X), \Lambda_{k}(j)[i]);
\end{equation}
\begin{equation}\label{eqn:motivic-Coh-defn-compact support}
  H^i_{M,c}(X, \Lambda(j)) = \Hom_{\dm_\cdh(k, \Lambda)}(M^c_k(X),  \Lambda_{k}(j)[i]).
\end{equation}
\end{defn}

\vspace{1mm}

By \cite[Corollary 5.9]{CD-Doc}, the change of topology functor $\dm(X, \Lambda) \rightarrow  \dm_\cdh(X, \Lambda)$ is an equivalence of monoidal triangulated categories if $X \in \Sm_{k}$. This implies that for $X \in \Sm_{k}$, the above cohomology group \eqref{eqn:motivic-Coh-defn} agrees with the usual Voevodsky motivic cohomology groups, i.e., 
\begin{equation} \label{eqn:cdh-mot=mot}
    H_{M}^{i}(X, \Lambda(j)) \cong \Hom_{\dm(k, \Lambda)} (M_{k}(X), \Lambda_{k}(j)[i] ).
\end{equation}

For $X  \in \Sm_{k}$, we let
\begin{center}
    $H_{M}^i(X, \Z(j)) :=
\Hom_{\dm(k, \Z)}(\Z_{\rm tr}(X), \Z(j)[i])$;
\end{center}
\begin{center}
    $H^i_{M,c}(Y, \Z(j)) :=
\Hom_{\dm(k, \Z)}(f_* f^! \Z_k, \Z(j)[i])$,
\end{center}
 where $\Z(1) =
{\Z_{\rm tr}(\P^1_k)}/{\Z_{\rm tr}(\Spec(k))}$ and $f \colon X \to \Spec(k)$ is the structure map. 
We then have the following isomorphism (see \cite[Corollary 2]{Voe-imrn}). 
\begin{equation}\label{motivic-higher-chow-iso}
  H_{M}^{i}(X, \mathbb{Z}(j)) \cong \CH^{j}(X, 2j-i),
\end{equation} 
   where $\CH^{j}(X,i)$ denotes Bloch's higher Chow groups (\cite{Bloch-Adv}, \cite[Definition 17.1]{Mazza-Voevodsky-Weibel}). 

For any  $X\in \Sch_{k} $, and any integer $n$ invertible on $X$, we have the following realization  maps (see  \cite[Equation(9.4), Lemma 9.3]{GKR}) for all $i\geq 0, j \in \Z$. 
\begin{equation}\label{eqn:RE-2}
 \epsilon^{i,j}_{X,n} :  H_{M}^i(X, \mathbb{Z}/n(j)) \rightarrow H^i_\et(X, \mathbb{Z}/n(j)).
\end{equation}

\begin{lem}\label{lem:Real-iso}
  Let $X \in \Sm_k$ be of dimension $d$. Then $\epsilon^{i,j}_{X,n}$ is an
  isomorphism
  when $j \ge {\min}\{i, d + cd(k)\}$. This map is injective when $i = j+1$.
\end{lem}
\begin{proof}
    See \cite[Lemma 10.1]{GKR}.
\end{proof}

 We denote $H_{M}^{i}(X, j) := \oplus_{l \neq p} H_{M}^{i}(X, \mathbb{Q}_{l}/\mathbb{Z}_{l}(j) )$,  where
 \begin{equation*}
     H_{M}^{i}(X, \mathbb{Q}_{l}/\mathbb{Z}_{l}(j) ) = \varinjlim_{v \in \mathbb{N} } H_{M}^{i}(X, \mathbb{Z}/l^{v}(j)).
 \end{equation*} The maps $\epsilon^{i,j}_{X,n}$ (see \ref{eqn:RE-2}) induce the  maps
    \begin{equation}\label{eqn:RE-2-infty}
 \epsilon^{i,j}_{X} :  H_{M}^i(X, j) \rightarrow H^i_\et(X, j).
 \end{equation}

\begin{lem}\label{lem:Real-iso infty}
  Let $X \in \Sm_k$ be of dimension $d$. Then the induced map $\epsilon^{i,j}_X$ is an
  isomorphism
  when $j \ge {\min}\{i, d + cd(k)\}$. This map is injective when $i = j+1$.
\end{lem}
  \begin{proof}
     This is an easy consequence of Lemma \ref{lem:Real-iso}.
 \end{proof}

\subsection{Finiteness of torsion over finite and local fields}\label{subsection:finiteness of torsion in higher chow groups}
Throughout this subsection, we let $k$ to be either a finite field of characteristic $p>0$  or a local field of residue characteristics $p >0$ and let $X$ be a smooth projective variety of dimension $d \geq 0$ over $k$. In this subsection, we prove finiteness of $p'$-torsion subgroup of $\CH^{i}(X, j)$ in certain ranges of $i$ and $j$. We let $\mathbb{L}$ be the set of natural numbers that are not divisible by $p$. For reader's convenience we restate Theorem~\ref{thm:Main-3}.

 \begin{thm}\label{thm:Main-3 restate}

       Let $X$ be a smooth projective variety of dimension $d$ over a local field $k$ having a good reduction. Then the $\CH^{i}(X, j) \lbrace p' \rbrace$ is finite if $j \geq 1$ and either $i \leq j+2$ or $i \geq d + 2$.  
        
        \end{thm}

 By (\ref{motivic-higher-chow-iso}), we have
\begin{equation}\label{motivic-chow-iso-p'-torsion}
     \CH^{i}(X, j) \lbrace p' \rbrace \cong H_{M}^{2i-j}(X, \mathbb{Z}(i) ) \lbrace p' \rbrace,
\end{equation}
  for integers $i$ and $j \geq 0$.
Therefore, finiteness of $\CH^{i}(X, j) \lbrace p' \rbrace$ is equivalent to the finiteness of  $H_{M}^{2i-j}(X, \mathbb{Z}(i) ) \lbrace p' \rbrace$. Furthermore,  using the exact triangle 
\begin{center}
    $\mathbb{Z}(i) \xrightarrow{n} \mathbb{Z}(i) \rightarrow \mathbb{Z}/n(i)$,
\end{center}
where $n \in \mathbb{L}$ be any integer, we get the following short exact sequence. 
\begin{equation}\label{eqn:H1-fin-3.0}
	0 \to H_{M}^{2i-j-1}(X, \mathbb{Z}(i) ) \otimes \mathbb{Z}/n \rightarrow H_{M}^{2i-j-1}(X, \mathbb{Z}/n(i))
	\rightarrow _nH_{M}^{2i-j}(X, \mathbb{Z}(i)) \to 0,
	\end{equation}
Taking limit over all integers $n$  in $\mathbb{L}$, we get
\begin{equation}\label{eqn:H1-fin-3}
	0 \to H_{M}^{2i-j-1}(X, \mathbb{Z}(i) ) \otimes (\mathbb{Q}/\mathbb{Z})' \rightarrow H_{M}^{2i-j-1}(X, i)
	\rightarrow H_{M}^{2i-j}(X, \mathbb{Z}(i)) \{ p'\} \to 0.
	\end{equation}

 \begin{prop}\label{prop: motivic finiteness bounded}
    Let $X$ be a smooth projective variety of dimension $d$ over $k$. 
     \begin{enumerate}
        \item If $k$ is finite, then  $H_{M}^{2i-j-1}(X, i )$ is finite if $j \geq 0$ and either $i \leq j+2$ or $i \geq d + 1$.

        \vspace{2mm}

        \item If $k$ is local, and $X$ admits a good reduction, then  $H_{M}^{2i-j-1}(X, i )$ is finite if $j \geq 1$ and either $i \leq j+2$ or $i \geq d + 2$. 

    \end{enumerate}
    
        \end{prop}

\begin{proof}
     By Lemma \ref{lem:Real-iso infty}, the map  $\epsilon^{2i-j-1,i}_{X} :  H_{M}^{2i-j-1}(X, i) \rightarrow H_{\et}^{2i-j-1}(X, i)$ is injective whenever either $i \leq j+2$ or $i \geq d + cd(k)$. Note that $cd(k) = 1$ and $2$ when $k$ is finite and local, respectively.
  On the other hand, the group $H_{\et}^{2i-j-1}(X, i)$ is finite if $ j \geq 0$ (resp. $j \geq 1$) by Theorem \ref{finiteness-smooth-finite field} (resp.  Corollary \ref{good-reduction-finiteness}). This completes the proof.
\end{proof}

 \begin{thm}\label{thm:p'-torsion finiteness in higher chow group}
    Let $X$ be a smooth projective variety of dimension $d$ over $k$. 
     \begin{enumerate}
        \item If $k$ is finite, then  $\CH^{i}(X, j) \lbrace p' \rbrace$ is finite if $j \geq 0$ and either $i \leq j+2$ or $i \geq d + 1$.

        \vspace{2mm}

        \item If $k$ is local, and $X$ admits a good reduction, then  $\CH^{i}(X, j) \lbrace p' \rbrace$ is finite if $j \geq 1$ and either $i \leq j+2$ or $i \geq d + 2$. 

    \end{enumerate}
    
        \end{thm}

        \begin{proof}
            By (\ref{motivic-chow-iso-p'-torsion}) and (\ref{eqn:H1-fin-3}), it is enough to show that $H_{M}^{2i-j-1}(X, i )$ is finite in the respective ranges. This follows from Proposition \ref{prop: motivic finiteness bounded}. This completes the proof of Theorem \ref{thm:p'-torsion finiteness in higher chow group}, and hence completes the proof of Theorem \ref{thm:Main-3 restate}.
        \end{proof}

This completes the proof of Theorem~\ref{finiteness for chow groups over finite fields} and Theorem~\ref{thm:Main-3}

\subsection{Finiteness of $p'$-torsion over higher local fields}\label{torsion over N-local fields}

In this section, we prove the finiteness of $p'$-torsion subgroup of the higher Chow groups over $N$-local fields when $N \geq 2$. We prove the following.
\begin{thm}\label{thm:finiteness of p'-torsion in chow groups over N-fields}
 Let $X$ be a smooth projective variety of dimension $d$ over $k_{N}$, where $N\geq 2$. Then  $\CH^{i}(X, j) \lbrace p' \rbrace$ is finite if $j \geq d+N$ and either $i \leq j+2$ or $i \geq d + N+1$. 
\end{thm}

\begin{proof}
     By (\ref{motivic-chow-iso-p'-torsion}) and (\ref{eqn:H1-fin-3}), it is enough to show that $H_{M}^{2i-j-1}(X, i )$ is finite in the given range. By Lemma \ref{lem:Real-iso infty}, the map  $\epsilon^{2i-j-1,i}_{X} :  H_{M}^{2i-j-1}(X, i) \rightarrow H_{\et}^{2i-j-1}(X, i)$ is injective whenever $i \leq j+2$ or $i \geq d + N+1$. On the other hand, the group $H_{\et}^{2i-j-1}(X, i)$ is finite if $ j \geq d+N$ by Corollary \ref{thm:finiteness for quasiprojective over N-local fields}. This completes the proof.
\end{proof}

\section{Application to class field theory} \label{sec:App-CFT}

In this section, we apply the finiteness result about torsion in the higher Chow groups to the  tame class field theory for smooth varieties over local fields. We fix $k$ to be a local field with the residue characteristic $p >0$ and let $X$ be a smooth variety over $k$ with a smooth compactification $\overline{X}$. In \cite{GKR}, the authors studied tame class field theory for smooth varieties over local fields. In loc. cit, it was  shown that the kernel of reciprocity map 
\[
\rho^{t}_{X} : C^{t}(X) \rightarrow \pi^{ab}_{1}(X)
\]
is $p'$-divisible, whenever $\overline{X}$ has a good reduction  \cite[Theorem 1.5]{GKR}, where $C^{t}(X)$ is idele class group associated to $X$ (see \cite[Definition 4.7]{GKR})  and $\pi^{ab}_{1}(X)$ is the abelianized étale fundamental group of $X$ . In this section, we strengthen this result by showing that the $\Ker(\rho^{t}_{X}) $ is in fact uniquely $p'$-divisible. In order to prove this, we prove certain key lemmas which are also of independent interest. Recall from \eqref{eqn:RE-2}, we have etale realization maps
 $\epsilon^{i,j}_{\overline{X},n} :  H_{M}^i(\overline{X}, \mathbb{Z}/n(j)) \rightarrow H^i_\et(\overline{X}, \mathbb{Z}/n(j))$ for integers $i, j$.

\begin{lem}\label{realisation_iso_good red}
Let $\overline{X}$ be smooth projective variety of dimension $d$ over $k$ having good reduction. Then étale realization maps $\epsilon_{ \overline{X},m}^{2d, d+1} : H_{M}^{2d}(\overline{X}, \mathbb{Z}/m(d+1) ) \rightarrow H_{\et}^{2d}(\overline{X}, \mathbb{Z}/m(d+1) )$ is an isomorphism, where $m \in \mathbb{L}$ be any integer. 
\end{lem}

\begin{proof}
	
	This map is known to be surjective (see \cite[Lemma 13.2]{GKR}). So it is enough to show this map is injective. Let's denote by $\mathcal{H}^{i}(\mu^{\otimes d+1}_{m})$ the Zariski sheaf associated to the presheaf $ U \rightarrow  H_{\et}^{i}(X, \mu^{\otimes d+1}_{m} )$ for every open $ U \subset X$.   We have the following distinguished  triangle in the derived category of complexes of Zariski sheaves (see \cite[Page 190]{Szamuely}).
	\begin{equation}
	\tau_{\leq d+1}\mathbb{Z}/m(d+1) \rightarrow R\epsilon_{\ast}\mu^{\otimes d+1}_{m} \rightarrow R^{d+2}\epsilon_{\ast}\mu^{\otimes d+1}_{m}[-d-2].
	\end{equation}

This yields a long exact sequence of cohomology groups.
\begin{equation}\label{les-Kato}
H_{zar}^{i-d-3}(\overline{X},  \mathcal{H}^{d+2}(\mu^{\otimes d+1}_{m})  ) \rightarrow  H^{i}_{zar}(\overline{X}, \tau_{\leq d+1}\mathbb{Z}/m(d+1) ) \rightarrow H^{i}_{\et}(\overline{X}, \mathbb{Z}/m(d+1) ) \rightarrow H_{zar}^{i-d-2}(\overline{X},  \mathcal{H}^{d+2}(\mu^{\otimes d+1}_{m})  )
\end{equation}

By Bloch-Ogus \cite{Bloch-Ogus}, we know the following complex (\ref{Bloch-ogus complex}) can be used to compute the Zariski cohomology of the sheaf $\mathcal{H}^{d+2}(\mu^{\otimes d+1}_{m})$.

\begin{equation}\label{Bloch-ogus complex}
 \bigoplus_{x \in \overline{X}_{(d)}} H^{d+2}_{\et}(k(x), \mu^{\otimes d+1}_{m})  \rightarrow \bigoplus_{x \in \overline{X}_{(d-1)}} H^{d+1}_{\et}(k(x), \mu^{\otimes d}_{m}) \rightarrow \cdots \rightarrow \bigoplus_{x \in \overline{X}_{(0)}} H^{2}_{\et}(k(x), \mu_{m}).
\end{equation}

We denote the above complex (\ref{Bloch-ogus complex}) by $C_{d+2, d+1}(\overline{X}, \mathbb{Z}/m)$. Therefore, we get that
\begin{equation}
H_{zar}^{i}(\overline{X},  \mathcal{H}^{d+2}(\mu^{\otimes d+1}_{m}) ) \cong H_{d-i} (C_{d+2, d+1}(\overline{X}, \mathbb{Z}/m)).
\end{equation}
Since the complex $C_{d+2, d+1}(\overline{X}, \mathbb{Z}/m)$  is known to be acyclic in positive degrees (see \cite[Theorem 0.4]{Kerz-Saito}), we have $H_{zar}^{i}(\overline{X},  \mathcal{H}^{d+2}(\mu^{\otimes d+1}_{m}) ) = 0$ whenever $ i \neq d$. It follows from sequence (\ref{les-Kato}) that the map $H^{2d}_{zar}(\overline{X}, \tau_{\leq d+1}\mathbb{Z}/m(d+1) ) \rightarrow H^{2d}_{\et}(\overline{X}, \mathbb{Z}/m(d+1) )$ is an isomorphism. We are therefore reduced  to show that 
the natural map  
\begin{equation}\label{motivic-iso-truncated complex}
     H^{2d}_{M}(\overline{X}, \mathbb{Z}/m(d+1) ) \to  H^{2d}_{zar}(\overline{X}, \tau_{\leq d+1}\mathbb{Z}/m(d+1) )
\end{equation} is injective.
To see this, we consider the following commutative diagram, where the top horizontal arrow is the same as the map in \eqref{motivic-iso-truncated complex}. 
\begin{equation}
    \xymatrix@C.8pc{ 
    H^{2d}_{M}(\overline{X}, \mathbb{Z}/m(d+1) ) \ar[r] \ar[d] & H^{2d}_{zar}(\overline{X}, \tau_{\leq d+1}\mathbb{Z}/m(d+1) ) \ar[d] \\
H^{2d}_{M}(\overline{X}^{\hs}, \mathbb{Z}/m(d+1) )  \ar[r] & H^{2d}_{zar}(\overline{X}^{\hs}, \tau_{\leq d+1}\mathbb{Z}/m(d+1) ).}
\end{equation}
The left vertical arrow is an isomorphism by \cite[Proposition 8.2]{GKR} and the bottom horizontal arrow is isomorphism by 
 \cite[Proposition 13.9]{Mazza-Voevodsky-Weibel} and \cite[Theorem 1.6]{GL-Crelle}. It follows that the top horizontal arrow is injective, which completes the proof.
\end{proof}

\begin{cor}\label{torsion in SK1}
	The group $\CH^{d+1}(\overline{X}, 1)\lbrace p' \rbrace$ is finite.
\end{cor}
\begin{proof}

By \eqref{motivic-chow-iso-p'-torsion}, we have 
\begin{equation}
     \CH^{d+1}(X, 1) \lbrace p' \rbrace \cong H_{M}^{2d+1}(X, \mathbb{Z}(d+1) ) \lbrace p' \rbrace.
\end{equation}
Using the short exact sequence \eqref{eqn:H1-fin-3} with $i = d+1$ and $j=1$, it suffices to show that the group $H_{M}^{2d}(X, d+1)$ is finite. We now apply Lemma \ref{realisation_iso_good red} to reduce the problem to showing that $H_{\et}^{2d}(X, d+1)$ is finite. Indeed, Lemma \ref{realisation_iso_good red} provides an isomorphism $H_{M}^{2d}(X, d+1) \cong H_{\et}^{2d}(X, d+1)$. The latter group is finite by Corollary \ref{good-reduction-finiteness}. This completes the proof.
	\end{proof}

\begin{prop}\label{finitess of torsion in tame class group}
	The group $C^{t}(X)\lbrace p' \rbrace$ is finite.
\end{prop}

\begin{proof}  Note that $C^{t}(X)\lbrace p' \rbrace $ injects into $C^{t}(X)[ \frac{1}{p} ] $. So it is enough to show that $C^{t}(X)[ \frac{1}{p} ] \lbrace p' \rbrace$ is finite. By \cite[Prop 8.15]{GKR}, we know that 
	
	\begin{equation}
 H_{c,M}^{2d+1}(X,\mathbb{Z}[1/p](d+1) ) \cong C^{t}(X)[
   \frac{1}{p}].
	\end{equation}

We have exact triangle of complexes,
\begin{equation}
0 \rightarrow \mathbb{Z}[\frac{1}{p}] \xrightarrow{m}  \mathbb{Z}[\frac{1}{p}] \rightarrow \mathbb{Z}[\frac{1}{p}]/m \rightarrow 0.
\end{equation}	
As $p\nmid m$, $ \mathbb{Z}[\frac{1}{p}]/m \cong \mathbb{Z}/m$. We therefore have 
\begin{equation}\label{ses-multiplication by m}
0 \rightarrow \mathbb{Z}[\frac{1}{p}] \xrightarrow{m}  \mathbb{Z}[\frac{1}{p}] \rightarrow \mathbb{Z}/m \rightarrow 0.
\end{equation}
	
Using \eqref{ses-multiplication by m}, we get an exact sequence of cohomology groups,
	\begin{equation}
	0 \to H_{c,M}^{2d}(X, \mathbb{Z}[1/p](d+1) )/m \rightarrow H_{c,M}^{2d}(X, \mathbb{Z}/m(d+1)) \rightarrow
	_mH_{c,M}^{2d+1}(X,\mathbb{Z}[1/p](d+1) )  \to 0.
	\end{equation}
	
Therefore, it is enough to show that
\begin{equation}
\lvert	H_{c,M}^{2d}(X, \mathbb{Z}/m(d+1)) \rvert \leq M
	\end{equation}
for all $m \in \mathbb{L}$ and for some integer $M \geq 0$. Now using the long exact sequence,
	\begin{equation}
... \rightarrow	H_{M}^{2d-1}(\overline{X} \setminus X, \mathbb{Z}/m(d+1)) \rightarrow	H_{c,M}^{2d}(X, \mathbb{Z}/m(d+1)) \rightarrow  	H_{M}^{2d}(\overline{X}, \mathbb{Z}/m(d+1)) \rightarrow,
	\end{equation}
		we are reduced to show the existence of two integers $M'$ and $M''$ such that the following two conditions hold.

\begin{enumerate}

\vspace{1mm}
    \item $ \lvert  H_{M}^{2d-1}(\overline{X} \setminus X, \mathbb{Z}/m(d+1)) \rvert \leq M'$, 
    \vspace{1mm}

    \item $ \lvert	H_{M}^{2d}(\overline{X}, \mathbb{Z}/m(d+1)) \rvert \leq M''$.
\end{enumerate}
	The result for $\overline{X} \setminus X$ follows from \cite[Corollary 10.11]{GKR} and for $\overline{X}$, it follows from combining Lemma \ref{realisation_iso_good red} and Proposition \ref{etale cohomology bounded for finite and local}. This completes the proof.
    \end{proof}

We are now ready to prove the main result of this section, i.e., the kernel of tame reciprocity map $\rho^{t}_{X}$ is uniquely $p'$-divisible.

\begin{thm}
	The kernel of the tame reciprocity map $\rho^{t}_{X} : C^{t}(X) \rightarrow \pi^{ab,t}_{1}(X)$ is uniquely $p'$-divisible.
\end{thm}
\begin{proof}
	Note that $\Ker(\rho_{X}^{t})$ is $p'$-divisible by \cite[Theorem 13.3]{GKR}. It is therefore enough to show that $\Ker(\rho_{X}^{t}) \lbrace p' \rbrace = 0 $. Note that the group $\Ker(\rho_{X}^{t}) \lbrace p' \rbrace$ is $p'$-divisible because the group $\Ker(\rho_{X}^{t})$ itself is $p'$-divisible. Therefore, it suffices to show that $\Ker(\rho_{X}^{t}) \lbrace p' \rbrace$ is finite, but this follows from  Corollary \ref{finitess of torsion in tame class group}. This completes the proof.	
\end{proof}

\noindent\emph{Acknowledgements.} The second and the third authors would like to thank the Institute of Mathematical Sciences (IMSc), Chennai, for hosting their visits in June--July 2025, during which this work was completed.

\end{document}